\documentclass[12pt]{article}

\usepackage{graphicx, amssymb, latexsym, amsfonts, amsmath, lscape, amscd,
amsthm, color, epsfig, mathrsfs, tikz, enumerate, multirow}
\tikzset{XOR/.style={draw,circle,append after command={
        [shorten >=\pgflinewidth, shorten <=\pgflinewidth,]
        (\tikzlastnode.north) edge (\tikzlastnode.south)
        (\tikzlastnode.east) edge (\tikzlastnode.west)
        }
    }
}

\newtheorem{theorem}{Theorem}[section]
\newtheorem{conjecture}[theorem]{Conjecture}
\newtheorem{corollary}[theorem]{Corollary}

\newtheorem{lemma}[theorem]{Lemma}

\newtheorem{question}[theorem]{Question}

\theoremstyle{definition}
\newtheorem{definition}[theorem]{Definition}

\theoremstyle{definition}
\newtheorem{remark}[theorem]{Remark}
\newtheorem{note}[theorem]{Note}

\newcommand\DELETE[1]{}

\usetikzlibrary{decorations.pathmorphing}

\pgfdeclaredecoration{Snake}{initial}
{
  \state{initial}[switch if less than=+.625\pgfdecorationsegmentlength to final,
                  width=+.3125\pgfdecorationsegmentlength,
                  next state=down]{
    \pgfpathmoveto{\pgfqpoint{0pt}{\pgfdecorationsegmentamplitude}}
  }
  \state{down}[switch if less than=+.8125\pgfdecorationsegmentlength to end down,
               width=+.5\pgfdecorationsegmentlength,
               next state=up]{
    \pgfpathcosine{\pgfqpoint{.25\pgfdecorationsegmentlength}{-1\pgfdecorationsegmentamplitude}}
    \pgfpathsine{\pgfqpoint{.25\pgfdecorationsegmentlength}{-1\pgfdecorationsegmentamplitude}}
  }
  \state{up}[switch if less than=+.8125\pgfdecorationsegmentlength to end up,
             width=+.5\pgfdecorationsegmentlength,
             next state=down]{
    \pgfpathcosine{\pgfqpoint{.25\pgfdecorationsegmentlength}{\pgfdecorationsegmentamplitude}}
    \pgfpathsine{\pgfqpoint{.25\pgfdecorationsegmentlength}{\pgfdecorationsegmentamplitude}}
  }
  \state{end down}[width=+.3125\pgfdecorationsegmentlength,
                   next state=final]{
     \pgfpathcosine{\pgfqpoint{.15625\pgfdecorationsegmentlength}{-.5\pgfdecorationsegmentamplitude}}
     \pgfpathsine{\pgfqpoint{.15625\pgfdecorationsegmentlength}{-.5\pgfdecorationsegmentamplitude}}
  }
  \state{end up}[width=+.3125\pgfdecorationsegmentlength,
                 next state=final]{
     \pgfpathcosine{\pgfqpoint{.15625\pgfdecorationsegmentlength}{.5\pgfdecorationsegmentamplitude}}
     \pgfpathsine{\pgfqpoint{.15625\pgfdecorationsegmentlength}{.5\pgfdecorationsegmentamplitude}}
  }
  \state{final}{\pgfpathlineto{\pgfpointdecoratedpathlast}}
}

\begin{document}


\title{\bf A theoretical expansion of the {\sc Sprout} game}
\author{
{\sc Soura Sena Das}$\,^{a}$, {\sc Zin Mar Myint}$\,^{b}$, {\sc Soumen Nandi}$\,^{c}$, \\ {\sc Sagnik Sen}$\,^{b}$,  {\sc \'Eric Sopena}$\,^{d}$\\
\mbox{}\\
{\small $(a)$ Indian Statistical Institue, Kolkata, India.}\\
{\small $(b)$ Indian Institute of Technology Dharwad, India.}\\
{\small $(c)$ Institute of Engineering and Management, Kolkata, India.}\\
{\small $(d)$ Univ. Bordeaux, CNRS, Bordeaux INP, LaBRI, UMR 5800, F-33400, Talence, France}\\
}

\date{\today}

\maketitle

\begin{abstract}
{\sc Sprout} is a two-player pen and paper game 
which starts with $n$ vertices, and the players take 
turns to join two pre-existing dots by a subdivided 
edge while keeping the graph sub-cubic planar at all 
times. The first player not being able to move 
loses. A major conjecture claims that Player 1 has a 
winning strategy if and only if $n \equiv 3,4,5$ 
($\bmod~6$). The conjecture is verified until $44$, 
and a few isolated values of $n$, usually with the 
help of a computer. However, to the best of our 
understanding, not too much progress could be made towards finding a theoretical proof of the conjecture till now. 

In this article, we try to take a bottom-up approach and start building a theory around the problem. We start by expanding a related game called 
{\sc Brussels Sprout} (where dots are replaced by crosses) introduced by Conway, possibly to help the understanding of {\sc Sprout}. In particular, we introduce and study a generalized version of 
{\sc Brussels Sprout} where crosses are replaced by 
a dot having an arbitrary number of ``partial edges'' (say, general cross) coming out, and planar graphs are replaced by any (pre-decided) hereditary class of graphs. 
We study the game for forests, graphs on surfaces, and sparse planar graphs. We also do a nimber characterization of the game when the hereditary class is taken to be triangle-free planar graphs, and we have started the game with two arbitrary generalized crosses. Moreover, while studying this
particular case, we naturally stumble upon a circular version of the same game and solve a difficult nimber characterization using the method of structural induction. The above mentioned proof may potentially be one approach 
to solving the {\sc Sprout} conjecture. \end{abstract}


\section{Introduction}\label{intro}
In 1967, Conway and Paterson~\cite{gardner1967mathematical} introduced the two-player pen and paper game called {\sc Sprout}. The game starts with $n$ dots (vertices) on paper, and the players make their moves alternately. A valid move consists of connecting any dot to itself or to another dot with a curve (edge) and 
then placing a new dot on the curve drawn (subdivision).  
There are two conditions that need to be maintained during a move: the curve should not cross itself or any other curve (planarity), 
and a spot can have at most three 
curves incident to it (degree three). 
The first player that cannot make a move, loses  
(see Fig.~\ref{sproutgame} for an example). An enigmatic famous conjecture claims the winning strategy in this game~\cite{applegatecomputer}.

  \begin{conjecture}\label{conj sprouts}
  If the game of {\sc Sprout} starts with $n$ dots, then Player $1$ has a winning strategy if and only if $n \equiv 3, 4, 5\ (\bmod \ 6)$. 
  \end{conjecture}

To date, it has been possible to verify the correctness of 
the conjecture when the initial number of spots is 
$n \in \{1,2, \cdots, 44, 46, 47, 53\}$ 
(see~\cite{browne2016algorithms}). The verifications done are mostly 
using brute force case analysis~\cite{berlekamp2018winning}, and computer checks~\cite{applegatecomputer}. Therefore, finding theoretical tools to attack the conjecture is a natural open question. There have been a few approaches to attack the conjecture which can be summarized in few major points. 

\begin{itemize}
    \item A brute force analysis of the game starting with $n$ points: theoretical case analysis by hand~\cite{allen2004winning} for small values of $n$, and computer checks for large values of $n$~\cite{applegatecomputer}. The computer checks for larger values of $n$ were helped by smart theoretical tricks involving the analysis of nimber values for partially played games~\cite{lemoine2010computer,lemoine2012nimbers}. 
    
    \item Interesting properties of the game are explored in~\cite{browne2016algorithms,copper1993graph,focardi2004modular,lam1997connected}. Moreover, 
    a AI-based game play of {\sc Sprouts} is developed in~\cite{vcivzek2021implementation} which can play the game perfectly up to $11$ spots (initial condition).

      \item The game {\sc Sprout} was introduced as a game played on the plane (or the sphere) surface. Later the game was extended and analysed for 
    general surfaces, that is, orientable and 
    non-orientable surfaces having genus 
    $g \geq 0$~\cite{lemoine2008sprouts,lemoine2012nimbers}. 

    \item Analysis of similar games such as {\sc Brussels Sprout}~\cite{bezdek2023conway,cairns2007brussels} and {\sc Planted Brussels Sprout}~\cite{ji2018brussels, williams2018planted}. This approach finds its relevance by expressing partially played games.
\end{itemize}

Our approach is a blend of analysing similar games from the theoretical perspective, and thereby extending the knowledge base, and also by analysing the nimber of such games. Moreover, we present a nimber characterization of a 
non-trivial game using the method of structural induction, showing how it can become a potential tool to attack Conjecture~\ref{conj sprouts}. 

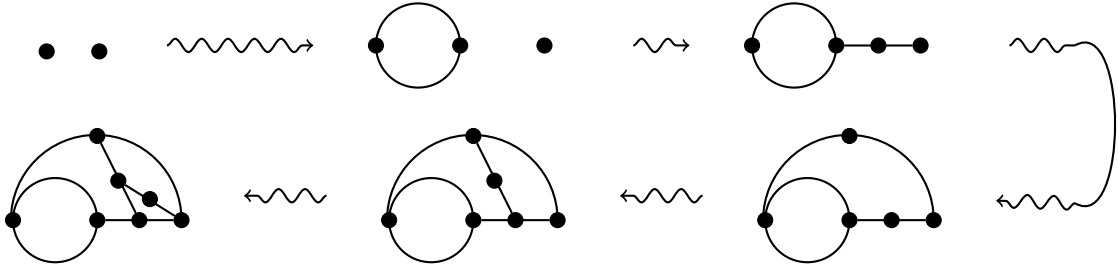
\begin{figure}
\begin{tikzpicture}[node distance=5cm]

\node (step1){\begin{tikzpicture}[scale=0.7]
\node[circle, draw,inner sep=0.1mm, minimum size=2mm, fill] (c7) at (0,-0.5) {};
\node[circle, draw,inner sep=0.1mm, minimum size=2mm, fill] (c7) at (1,-0.5) {};
\node[white,circle, draw,inner sep=0.1mm, minimum size=1mm, fill] (c7) at (2,-0.2) {};
\end{tikzpicture}};

\node (step2) [right of=step1] {\begin{tikzpicture}[scale=0.7]
\node[circle, draw,inner sep=0.1mm,minimum size=2mm, fill] (c1) at (-1.6,-0.5) {};
\node[circle, draw,inner sep=0.1mm,minimum size=2mm, fill] (c2) at (0,-0.5) {};

\draw[thick] (c2) arc[start angle=0, end angle=360, radius=0.8];
\node[circle, draw,inner sep=0.1mm,minimum size=2mm, fill] (c3) at (1.6,-0.5) {};
\node[white,circle, draw,inner sep=0.1mm, minimum size=1mm, fill] (c0) at (-2.5,-0.2) {};
\node[white,circle, draw,inner sep=0.1mm, minimum size=1mm, fill] (c7) at (3,-0.2) {};
\end{tikzpicture}};

\node (step3) [right of=step2] {\begin{tikzpicture}[scale=0.7]
\node[circle, draw,inner sep=0.1mm,minimum size=2mm, fill] (c1) at (-1.6,-0.5) {};
\node[circle, draw,inner sep=0.1mm,minimum size=2mm, fill] (c2) at (0,-0.5) {};
\draw[thick] (c2) arc[start angle=0, end angle=360, radius=0.8];
\node[circle, draw,inner sep=0.1mm,minimum size=2mm, fill] (c3) at (1.6,-0.5) {};
\draw[thick] (c2)--(c3);
\node[circle, draw,inner sep=0.1mm,minimum size=2mm, fill] (c4) at (0.8,-0.5) {};
\node[white,circle, draw,inner sep=0.1mm, minimum size=1mm, fill] (c7) at (-2.5,-0.2) {};
\node[white,circle, draw,inner sep=0.1mm, minimum size=1mm, fill] (c7) at (3,-0.2) {};
\end{tikzpicture}};

\node (step4) [below of=step3, yshift=3cm] {\begin{tikzpicture}[scale=0.7]
\node[circle, draw,inner sep=0.1mm,minimum size=2mm, fill] (c1) at (-1.6,0) {};
\node[circle, draw,inner sep=0.1mm,minimum size=2mm, fill] (c2) at (0,0) {};
\draw[thick] (c2) arc[start angle=0, end angle=360, radius=0.8];
\node[circle, draw,inner sep=0.1mm,minimum size=2mm, fill] (c3) at (1.6,0) {};

\draw[thick] (c2)--(c3);
\draw[thick] (c3) arc[start angle=0, end angle=180, radius=1.62];
\node[circle, draw,inner sep=0.1mm,minimum size=2mm, fill] (c4) at (0.8,0) {};
\node[circle, draw,inner sep=0.1mm,minimum size=2mm, fill] (c5) at (0,1.6) {};
\node[white,circle, draw,inner sep=0.1mm, minimum size=1mm, fill] (c7) at (-2.5,-0.2) {};
\node[white,circle, draw,inner sep=0.1mm, minimum size=1mm, fill] (c7) at (2.5,-0.2) {};
\end{tikzpicture}};

\node (step5) [left of=step4] {\begin{tikzpicture}
[scale=0.7]
\node[circle, draw,inner sep=0.1mm,minimum size=2mm, fill] (c1) at (-1.6,0) {};
\node[circle, draw,inner sep=0.1mm,minimum size=2mm, fill] (c2) at (0,0) {};
\draw[thick] (c2) arc[start angle=0, end angle=360, radius=0.8];
\node[circle, draw,inner sep=0.1mm,minimum size=2mm, fill] (c3) at (1.6,0) {};

\draw[thick] (c2)--(c3);
\draw[thick] (c3) arc[start angle=0, end angle=180, radius=1.62];
\node[circle, draw,inner sep=0.1mm,minimum size=2mm, fill] (c4) at (0.8,0) {};
\node[circle, draw,inner sep=0.1mm,minimum size=2mm, fill] (c5) at (0,1.6) {};   
\draw[thick] (c5) -- (c4);
\node[circle, draw,inner sep=0.1mm,minimum size=2mm, fill] (c6) at (0.4,0.75) {};
\node[white,circle, draw,inner sep=0.1mm, minimum size=1mm, fill] (c7) at (-2.5,0) {};
\node[white,circle, draw,inner sep=0.1mm, minimum size=1mm, fill] (c7) at (2.5,0) {};
\end{tikzpicture}};

\node (step6) [left of=step5] {\begin{tikzpicture}
[scale=0.7]
\node[circle, draw,inner sep=0.1mm,minimum size=2mm, fill] (c1) at (-1.6,0) {};
\node[circle, draw,inner sep=0.1mm,minimum size=2mm, fill] (c2) at (0,0) {};
\draw[thick] (c2) arc[start angle=0, end angle=360, radius=0.8];
\node[circle, draw,inner sep=0.1mm,minimum size=2mm, fill] (c3) at (1.6,0) {};

\draw[thick] (c2)--(c3);
\draw[thick] (c3) arc[start angle=0, end angle=180, radius=1.62];
\node[circle, draw,inner sep=0.1mm,minimum size=2mm, fill] (c4) at (0.8,0) {};
\node[circle, draw,inner sep=0.1mm,minimum size=2mm, fill] (c5) at (0,1.6) {};   
\draw[thick] (c5) -- (c4);
\node[circle, draw,inner sep=0.1mm,minimum size=2mm, fill] (c6) at (0.4,0.75) {};
\draw[thick] (c6) -- (c3);
\node[circle, draw,inner sep=0.1mm,minimum size=2mm, fill] (c6) at (1,0.4) {};
\node[white,circle, draw,inner sep=0.1mm, minimum size=1mm, fill] (c7) at (-2.5,0) {};
\node[white,circle, draw,inner sep=0.1mm, minimum size=1mm, fill] (c7) at (2.5,0) {};
\end{tikzpicture}};
\draw [thick, ->,decorate,decoration=snake] (step1) -- (step2);
\draw [thick, ->,decorate,decoration=snake] (step2) -- (step3);

\path [thick] (13,0) edge [bend left=120,looseness=1] (13,-2.1);

\draw[thick,decorate,decoration=snake](step3) -- (13,0);
\draw [thick, ->,decorate,decoration=snake] (13,-2.1)--(step4);
\draw [thick, ->,decorate,decoration=snake] (step4) -- (step5);
\draw [thick, ->,decorate,decoration=snake] (step5) -- (step6);

\end{tikzpicture}
\caption{An example of {\sc Sprout} game.} \label{sproutgame}
\end{figure}

\begin{remark}
    There are several instances in the history of graph theory and combinatorics research where a rich theory developed while attempting to build tools for solving 
    a particularly challenging open problem. For example, in attempt to solve the Four-Color Theorem 
    researchers developed important concepts of graph theory including vertex coloring, edge coloring, and chromatic polynomials~\cite{west2001introduction}. 
    In the field of combinatorial games, the study and analysis of the game {\sc Nim} gave rise to the theory of impartial games~\cite{albert2019lessons}. It is noteworthy, that the game {\sc Sprout} is also an impartial game. This paper attempts to initiate such a bottom up approach to develop a theory around the game of {\sc Sprout}. Our initial approach is to generalize and analyze the study of a related, and better understood game called 
    {\sc Brussels sprout} introduced by Conway~\cite{bezdek2023conway}. 
\end{remark}

\subsection{A generalization of the {\sc Brussels sprout} game}
Conway~\cite{gardner1967mathematical} introduced a variation of {\sc Sprout}, called {\sc Brussels sprout}, possibly as a potential way to approach the study of {\sc Sprout}. This is also a two player pen and paper game where, instead of dots, we start with $n$ crosses. Each cross has four open arms or \textit{open tips} and a player can only connect the open tips. So a valid move consists of connecting any two open tips with a curve, and a crossbar is placed on the newly made curve. The crossbar creates two new open tips which can be used in subsequent moves. 
This game also retains the restriction of the curves not crossing each other (planarity) during the play from {\sc Sprout}.
As a two-player game, each player makes their move on alternate turns, and the first player who cannot make a valid move loses.  However, in this case, using Euler's formula for planar graphs, one can easily figure out the player having 
a winning strategy.
In fact, the moves made by the players are redundant, and no matter how they play, the total number of moves 
and the winner of the game is a function of the number $n$ of the initial crosses only (see Fig.~\ref{Brusselgame} for an example).

In this article, we look into a generalized version of {\sc Brussels sprout}, where instead of crosses we have a variable number of open tips for each spot. 
A set of graphs $\mathcal{F}$ is called a \textit{hereditary class} if every graph isomorphic to an induced subgraph of a graph in $\mathcal{F}$ belongs to $\mathcal{F}$.
We also restrict the intermediate steps to belong to hereditary graph classes and study them. 

\begin{definition}[Generalized {\sc Brussels sprout}]
    Given a hereditary class $\mathcal{F}$, we define the game $n$-{\sc Brussels sprout} for $\mathcal{F}$ with parameters $(t_1, t_2, \cdots, t_n)$, denoting it as $BS_n(\mathcal{F}:t_1,t_2,t_3,\dots, t_n)$, as follows. The game 
$BS_n(\mathcal{F}:t_1,t_2,t_3,\dots, t_n)$ 
starts with $n$ spots, having $t_1, t_2, \cdots, t_n$ open tips, respectively. A valid move consists of joining two open tips with a curve followed by drawing a crossbar on the curve to create two new
open tips. The graph obtained by considering the spots and intersections of a curve and a crossbar as vertices, and the curves joining two such vertices as edges, must remain inside the family $\mathcal{F}$ at all times. The first player unable to provide a valid move, loses. 
\end{definition}

\subsection{Our contribution and organization of the article}
\begin{itemize}
    \item In Section~\ref{sec surface}, we  study 
the possible number of moves and winning strategies for $n$-{\sc Brussels sprout} for the families of forests, and graphs on orientable and non-orientable surfaces of genus $k \geq 0$.  We notice that the player who has a winning strategy does not change even if we increase the value of the genus for the orientable surfaces. This phenomenon is also noticed experimentally for the game {\sc Sprout}, however, probably our theoretical proof for {\sc Brussels Sprout} 
provides an explanation for the same.

\item In Section~\ref{sec sparse}, we focus on studying the generalized {\sc Brussels sprout} game for the hereditary families of sparse planar graphs. In particular, we study the game for the family $\mathcal{P}_g$ of planar graphs having girth at least $g$. One particular case, when $g=4$, attracts our attention and we study it in a deep way as usual.

\item In Section~\ref{sec circular}, we introduce a new, related game called  {\sc Circular sprout}. 
We explore a relation between a particular class of {\sc Circular sprout} and 
$BS_2(\mathcal{P}_4: p,q)$, where $\mathcal{P}_4$ is the planar graph with girth $g\geq 4$.

\item In Section~\ref{sec nimber of circular}, we present a full nimber characterization of the game {\sc Circular sprout}
$CS_4(\mathcal{P}_4: p,1,q,1)$. 
The proof is novel and uses the method of structural induction. 

\item In Section~\ref{sec nimber of BS}, we show that the nimber of $BS_4(\mathcal{P}_4: p,q)$ is always $0$ implies that the second player always has a winning strategy using the result from the section~\ref{Nimbersec}.

\item In Section~\ref{sec conclusion},  we conclude the article with discussions of several open questions. 
\end{itemize}

\begin{note}
     We will follow standard graph notation according to West~\cite{west2001introduction} throughout this article unless otherwise stated.
\end{note}

\begin{note}
     A preliminary version of this work accepted for presentation at ICGT 2022.
\end{note}

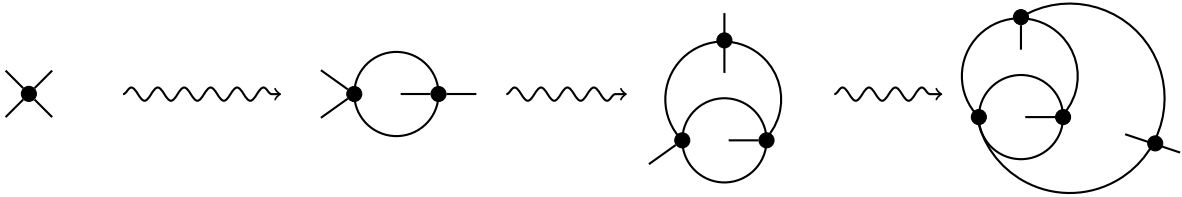
\begin{figure}
\begin{tikzpicture}[node distance=4.5cm]

\node (step1){\begin{tikzpicture}[scale=0.7]
\node[circle, draw,inner sep=0.1mm,minimum size=2mm, fill] (c7) at (0,0) {};
\node[white, circle, draw,inner sep=0.1mm,minimum size=1mm, fill] (c8) at (-0.5,-0.5) {};
\node[white,circle, draw,inner sep=0.1mm,minimum size=1mm, fill] (c9) at (-0.5,0.5) {};
\node[white,circle, draw,inner sep=0.1mm,minimum size=1mm, fill] (c10) at (0.5,0.5) {};
\node[circle, white,draw,inner sep=0.1mm,minimum size=1mm, fill] (c11) at (0.5,-0.5) {};
\node[circle, white,draw,inner sep=0.1mm,minimum size=1mm, fill] (d) at (1.5,-0.5) {};
\draw[thick] (c8) -- (c7);
\draw[thick] (c9) -- (c7);
\draw[thick] (c10) -- (c7);
\draw[thick] (c11) -- (c7);
\end{tikzpicture}};

\node (step2) [right of=step1] {\begin{tikzpicture}[scale=0.7]
\draw[thick] (c7) arc[start angle=0, end angle=360, radius=0.8];
\node[circle, draw,inner sep=0.1mm,minimum size=2mm, fill] (c7) at (0,0) {};
\node[white, circle, draw,inner sep=0.1mm,minimum size=1mm, fill] (c8) at (-0.8,0) {};
\node[white,circle, draw,inner sep=0.1mm,minimum size=1mm, fill] (c9) at (-2.3,0.5) {};
\node[white,circle, draw,inner sep=0.1mm,minimum size=1mm, fill] (c10) at (-2.3,-0.5) {};
\node[circle, draw,inner sep=0.1mm,minimum size=2mm, fill] (c12) at (-1.6,0) {};
\node[white,circle, draw,inner sep=0.1mm,minimum size=1mm, fill] (c13) at (0.8,0) {};
\node[circle, white,draw,inner sep=0.1mm,minimum size=1mm, fill] (d2) at (-2.7,-0.5) {};
\node[circle, white,draw,inner sep=0.1mm,minimum size=1mm, fill] (d1) at (1,-0.5) {};
\draw[thick] (c10) -- (c12);
\draw[thick] (c9) -- (c12);
\draw[thick] (c8) -- (c7);
\draw[thick] (c13) -- (c7);

\end{tikzpicture}};

\node (step3) [right of=step2] {\begin{tikzpicture}[scale=0.7]
\node[circle, draw,inner sep=0.1mm,minimum size=2mm, fill] (c7) at (0,0) {};
\draw[thick] (c7) arc[start angle=0, end angle=360, radius=0.8];
\node[circle, draw,inner sep=0.1mm,minimum size=2mm, fill] (c12) at (-1.6,0) {};
\node[white, circle, draw,inner sep=0.1mm,minimum size=1mm, fill] (c8) at (-0.8,0) {};
\node[white,circle, draw,inner sep=0.1mm,minimum size=1mm, fill] (c9) at (-0.8,2.5) {};
\node[white,circle, draw,inner sep=0.1mm,minimum size=1mm, fill] (c10) at (-2.3,-0.5) {};
\node[white,circle, draw,inner sep=0.1mm,minimum size=1mm, fill] (c14) at (-0.8,1.2) {};
\draw[thick] (c12) arc[start angle=225, end angle=-50, radius=1.1];
\node[circle, draw,inner sep=0.1mm,minimum size=2mm, fill] (c6) at (-0.8,1.9) {};
\node[circle, white,draw,inner sep=0.1mm,minimum size=1mm, fill] (d2) at (-2.2,-0.5) {};
\node[circle, white,draw,inner sep=0.1mm,minimum size=1mm, fill] (d1) at (1,-0.5) {};

\draw[thick] (c10) -- (c12);
\draw[thick] (c9) -- (c14);
\draw[thick] (c8) -- (c7);

\end{tikzpicture}};

\node (step4) [right of=step3] {\begin{tikzpicture}[scale=0.7]
\draw[thick] (c7) arc[start angle=0, end angle=360, radius=0.8];
\node[circle, draw,inner sep=0.1mm,minimum size=2mm, fill] (c7) at (0,0) {};
\node[circle, draw,inner sep=0.1mm,minimum size=2mm, fill] (c12) at (-1.6,0) {};
\node[white, circle, draw,inner sep=0.1mm,minimum size=1mm, fill] (c8) at (-0.8,0) {};
\node[white,circle, draw,inner sep=0.1mm,minimum size=1mm, fill] (c14) at (-0.8,1.2) {};
\node[circle, draw,inner sep=0.1mm,minimum size=2mm, fill] (c3) at (1.75,-0.5) {};
\node[circle, draw,inner sep=0.1mm,minimum size=2mm, fill] (c6) at (-0.8,1.9) {};
\draw[thick] (c6) arc[start angle=121, end angle=-169, radius=1.8];
\draw[thick] (c12) arc[start angle=225, end angle=-50, radius=1.1];
\node[white, circle, draw,inner sep=0.1mm,minimum size=1mm, fill] (c4) at (2.3,-0.7) {};
\node[white, circle, draw,inner sep=0.1mm,minimum size=1mm, fill] (c5) at (1.1,-0.3) {};
\node[circle, white,draw,inner sep=0.1mm,minimum size=1mm, fill] (d2) at (-2,-0.5) {};
\draw[thick] (c4) -- (c5);
\draw[thick] (c6) -- (c14);
\draw[thick] (c8) -- (c7);

\end{tikzpicture}};
\draw [thick, ->,decorate,decoration=snake] (step1) -- (step2);
\draw [thick, ->,decorate,decoration=snake] (step2) -- (7.6, 0);
\draw [thick, ->,decorate,decoration=snake] (step3) -- (step4);

\end{tikzpicture}
\caption{An example of a {\sc Brussels Sprout} game.} \label{Brusselgame}
\end{figure}

\section{Forests and graphs on surfaces}\label{sec surface}

To begin the study, let us first consider the family of forests.

\begin{theorem}\label{th forest}
Let $\mathcal{F}_t$ be the family of forests. Then $BS_n(\mathcal{F}_t :t_1,t_2,..., t_n$) ends after exactly $n-1$ moves. 
\end{theorem}

\begin{proof}
Since $\mathcal{F}_t$ is the family of forests, the resultant graph of terminated $BS_n(\mathcal{F}:t_1,t_2,..., t_n$) must be a tree. That is, it is possible to make a move until we create a tree. On the other hand, once we create a tree through our game, it is not possible to make any other move, as it will create a cycle.

Suppose that the game is terminated after $x$ moves and that $G$ is the resultant graph. Therefore, $|V(G)| = (n+x)$ as we started with $n$ vertices and in each move we have added exactly one vertex. Also, $|E(G)| = 2x$ as we started with no edges and in each move we have added exactly two edges. We know that $G$ is a tree, and therefore, 
$|E(G)| = |V(G)| -1$. Hence we have $2x = (n+x)-1$, which implies, $x = n-1$. 
\end{proof}

Next, we will move our attention to the family $\mathcal{O}_k$ of graphs that can be drawn on orientable surfaces of genus $k$ without crossing.

\begin{theorem}\label{th orientable surface}
Let $\mathcal{O}_k$ be the family of graphs that can be drawn on orientable surfaces of genus $k$ without crossing.
Then the only possible numbers of moves until the game $BS_n(\mathcal{O}_k:t_1,t_2,...,t_n)$ terminates are 
$$(n-2) + 2j + \sum_{i=1}^n t_i,$$
where $j = 0, 1, \cdots, k$. 
\end{theorem}

\begin{proof}
Suppose the game ends after $x$ moves and let the resultant graph after the end of the game be $G$. 
Thus, $|V(G)|= n+x$ and $|E(G)| = 2x$ as we start with $n$ vertices, $0$ edges,  and include exactly one vertex and two edges in each move.

Furthermore, we observe that the game cannot end if a particular face contains two or more open tips, while in the last move involved in creating a particular face of $G$ will ensure at least one open tip inside the face. Thus, the number of open tips is equal to the number of faces of $G$. 
Moreover, in each move we close exactly two open tips, and create exactly two open tips. Thus, the number of open tips 
remain invariant throughout the game. Therefore, the number of faces in $G$ is equal to  $|F(G)| = \sum_{i=1}^n t_i$.

Notice that, even though $G$ is embedded on $\mathcal{O}_k$, it may be possible to embed  it on an orientable surface of genus less than $k$. 
Let $j$ be the least number for which $G$ can be embedded on $\mathcal{O}_j$. Thus, $G$ will satisfy the Euler's formula 
$$|V(G)| -|E(G)|+ |F(G)|=2-2j$$
for orientable surfaces.

Thus, by replacing the values of $|V(G)|$, $|E(G)|$ and $|F(G)|$ we get
\begin{equation*}
x+n-2x+\sum_{i=1}^n t_i= 2-2j
\implies
    x=(n-2) + 2j + \sum_{i=1}^n t_i
\end{equation*}
which completes the proof. 
\end{proof}

Recall that the orientable surface with genus $0$ is nothing but the sphere, and thus, the family $\mathcal{O}_0$ of graphs are nothing but planar graphs. The above theorem characterizes all possible number of moves for the game 
$BS_n(\mathcal{O}_k: t_1, t_2, \cdots, t_n)$ to end. 
Clearly, when $k=0$, that is, for the game 
$BS_n(\mathcal{O}_0: t_1, t_2, \cdots, t_n)$ the game will end after exactly 
$(n-2) + \sum_{i=1}^n t_i$
moves. Therefore, for planar graphs, the game will end after a constant number of moves, and the first player will win if and only if that constant is odd
irrespective of how the game gets played. 

\begin{corollary}\label{corollary}
The game $BS_n(\mathcal{O}_0: t_1, t_2, \cdots, t_n)$
will end exactly after $(n-2) + \sum_{i=1}^n t_i$ moves and the first player will win if and only if 
$n + \sum_{i=1}^n t_i$ is odd. 
\end{corollary}

\begin{proof}
Follows directly from Theorem~\ref{th orientable surface} by restricting it for $k=0$.
\end{proof}

On the other hand, if we consider the game 
$BS_n(\mathcal{O}_k: t_1, t_2, \cdots, t_n)$ for all $k \geq 1$, even though the number of moves after which the game may end is not a constant, note that 
it only differs by an even number. Thus we have the following.

\begin{corollary}\label{O_k}
In the game $BS_n(\mathcal{O}_k: t_1, t_2, \cdots, t_n)$ for $k \geq 1$,
the first player will win if and only if 
$n + \sum_{i=1}^n t_i$ is odd. 
\end{corollary}

\begin{proof}
Follows directly from Theorem~\ref{th orientable surface} by observing that $(n-2) + 2j + \sum_{i=1}^n t_i$ is odd if and only if $n + \sum_{i=1}^n t_i$ is odd for all $j = 0, 1, \cdots, k$. 
\end{proof}

On a similar vein, we also study the family $\mathcal{N}_k$ of 
graphs that can be drawn on non-orientable surfaces of genus 
$k$ without crossing. 

\begin{theorem}\label{th non-orientable surface}
Let $\mathcal{N}_k$ be the family of graphs that can be drawn on non-orientable surfaces of genus $k$ without crossing. 
Then the only possible numbers of moves until the game $BS_n(\mathcal{N}_k : t_1,t_2,...,t_n)$ terminates are 
$$(n-2) + j + \sum_{i=1}^n t_i,$$
where $j = 0, 1, \cdots, k$. 
\end{theorem}

\begin{proof}
Suppose the game ends after $x$ moves and let the resultant graph after the end of the game be $G$. Thus, we can observe that $G$ has
$|V(G)|= n+x$, $|E(G)| = 2x$ and $|F(G)| = \sum_{i=1}^n t_i$
using the same arguments as in the proof of 
Theorem~\ref{th orientable surface}. 

Notice that, even though $G$ is embedded on $\mathcal{N}_k$, it may be possible to embed it on a non-orientable surface of genus less than $k$. 
Let $j$ be the least number for which $G$ can be embedded on $\mathcal{N}_j$.
Therefore, $G$ satisfies the Euler's formula is $$|V(G)| -|E(G)|+ |F(G)|=2-j.$$

Thus, by replacing the values of $|V(G)|$, $|E(G)|$ and $|F(G)|$ we get 
\begin{equation*}
x+n-2x+\sum_{i=1}^n t_i= 2-j \implies
    x=(n-2) + j + \sum_{i=1}^n t_i
\end{equation*}
which completes the proof.
\end{proof}

Recall that the non-orientable surface with genus $0$ is nothing but the projective plane, and thus, the family $\mathcal{N}_0$ of graphs are nothing but the projective planar graphs. Therefore the following corollary follows directly. 

\begin{corollary}
The game $BS_n(\mathcal{N}_0: t_1, t_2, \cdots, t_n)$
will end exactly after $(n-2) + \sum_{i=1}^n t_i$ moves and the first player will win if and only if 
$n + \sum_{i=1}^n t_i$ is odd. 
\end{corollary}

\begin{proof}
Follows directly from Theorem~\ref{th non-orientable surface} by restricting it for $k=0$.
\end{proof}

However, in this case, for higher genus, unlike in the case of orientable surfaces, the parity of the number of moves after which the game may end is not the same. Therefore, depending on how the game is played, it may be won by Player 1 or Player 2. Thus, it makes sense to analyse winning strategies. We pose this as an open question.

\begin{question}
Which player has a winning strategy for the game 
$BS_n(\mathcal{N}_k: t_1, t_2, \cdots, t_n)$ when $k \geq 1$? 
\end{question}

\section{Sparse planar graphs}\label{sec sparse}
In this section, let us focus on the family
$\mathcal{P}_g$ of planar graphs with girth at least $g$. 
The first result shows that if we fix a particular value of $n$, then the number of moves after which the game  $BS_n(\mathcal{P}_g:t_1,t_2,...,t_n)$ ends is a constant for large enough values of $g$. 

\begin{theorem}{\label{girth and tree}}
Let $\mathcal{P}_g$ be the family of planar graphs having girth at least $g$. Then the game
$BS_n(\mathcal{P}_g: t_1, t_2, \cdots, t_n$) game ends exactly after $n-1$ moves for all $g \geq 2n+1$.
\end{theorem}

\begin{proof}
Let $G_x$ be the resultant graph after $x$ number of moves.
As we start with $n$ vertices and add one vertex in each move, 
we have $|V(G_x)|=(n+x)$. Also, we start with zero edges and add two edges in each move. Thus we have $|E(G_x)| = 2x$. 

Thus after $n$ moves, we have $2n$ vertices and $2n$ edges in $G_n$. This graph must have a cycle, and as the graph has only $2n$ vertices, the cycle cannot have length greater than or equal to $2n+1$. 
This is a contradiction. 

Hence, the number of moves cannot be more than $n-1$. On the other hand, $n-1$ moves are also the minimum number of moves due to Theorem~\ref{th forest}. 
\end{proof}

Next, we focus particularly on the family of triangle-free planar graphs, that is, $\mathcal{P}_4$. We find upper and lower bounds of the number of moves after which the game 
$BS_n(\mathcal{P}_4:t_1,t_2,...,t_n)$ ends.

\begin{theorem}{\label{theorem1.8}}
The number of moves after which the game $BS_n(\mathcal{P}_4:t_1,t_2,...,t_n)$ ends is between $(4+n)$  and $(n-2) + \sum_{i=1}^n t_i$, where $n \geq 2$ and $t_i \geq 3$. 
\end{theorem}

\begin{proof}
Suppose the game ends after $x$ moves and let the resultant graph after the end of the game be $G$ which is a planar graph, in particular. Thus, due to  Corollary~\ref{corollary} we have $x \leq (n-2) + \sum_{i=1}^n t_i$.

As we start with $n$ vertices and add one vertex in each move, 
we have $|V(G_x)|=(n+x)$. Also, as we start with zero edges and add two edges in each move, we have $|E(G_x)| = 2x$. 
Furthermore, since $t_i \geq 3$ and $n \geq 2$, 
we are forced to have $|F(G)| \geq 6$
which can be observed through a brute force case analysis that we omit. 

Hence by Euler's Formula, we have
\begin{align*}
|V(G)|-|E(G)|+|F(G)|=2 &\implies 
(x+n)-2x+6\leq 2\\
&\implies x\geq 4+n
\end{align*}
Therefore, $(4+n) \leq x \leq (n-2) +\sum_{i=1}^n t_i$.  
\end{proof}

A natural question we can ask here is whether there is a play 
of $BS_n(\mathcal{P}_4:t_1,t_2,...,t_n)$ that ends after 
$(n-2) + \sum_{i=1}^n t_i$
moves and one that ends after $(4+n)$ moves. In the next result, we will see that indeed for $n=2$, such plays exist when the ratio of $p$ and $q$ is at most two.

\begin{theorem}
There exists plays of $BS_2(\mathcal{P}_4:p,q)$ which ends after 
$(p+q)$ and $6$ moves, respectively, for $p \leq q \leq 2p$. 
\end{theorem}

\begin{proof}
Let the two vertices present in the initial stage of the game 
$BS_2(\mathcal{P}_4 : p,q)$ be $x$ and $y$ positioned on a horizontal line, $x$ 
being on the left side of $y$. Furthermore, let $x_1, x_2, ..., x_p$ be the open
tips coming out of $x$, arranged in a clockwise order around $x$ and 
let $y_1, y_2, ..., y_q$ be the open tips coming out of $y$, arranged in an 
anti-clockwise order around $y$.

First, we are going to describe the play of $BS_2(\mathcal{P}_4 : p,q)$ which 
ends after $(p+q)$ moves for $q=(p+r)$ where $0\leq r\leq p$. Observe that it is
enough to describe the required sequences of the $(p+q)$ moves in the play. 

Let the first move be connecting $x_1$ to $y_1$ with a curve and putting the crossbar $t_1$ on it. In the subsequent moves, we connect $x_i$ to $y_{2i-1}$ with a curve and put the crossbar $t_i$ on it for $i=2,3,\cdots, r+1$.  After that we connect $x_j$ to $y_{r+j}$ with a curve and put the crossbar $t_j$ for $j=r+2,r+3,\cdots,p$. 
That means, we have made a total of $p$ moves till now.

Next, we connect 
$t_{i}$ to $t_{i+1}$ with a curve and put a crossbar $s_{i}$
on it for $i=1,2,\cdots,p-1$, which are $(p-1)$ more moves. 
Then we connect $s_i$ with $y_{2i}$ with a curve and put a crossbar on it for $i = 1, 2, \cdots, r$, which amounts to $r$ moves. 
Finally, we connect $t_1$ with $t_p$ with a curve and put a crossbar on it. Observe that, this ends the play as no more moves can be made and a total of $p+(p-1)+r+1=2p+r=p+q$ moves are made. Thus we are done with the first part of the proof.

Secondly, we  are going to describe the play of 
$BS_2(\mathcal{P}_4 : p,q)$ which ends after $6$ moves. Observe that it is enough to describe the required sequences of the $6$ moves in the play.

The first two moves in this case are connecting $x_1$ to $y_1$ and $x_2$ to $y_2$ with two curves and putting the crossbars $t_1$ and $t_2$ on them, respectively. 
Notice that the above two curves divide the plane into two regions: $R_1$ containing the open tip $x_3$, $R_2$ not containing it. 
We connect $t_1$ and $t_2$ with a curve through $R_1$ and put a crossbar $t_3$ on it. Next, we connect $x_p$ and $y_3$ to  $t_3$ two curves and put crossbars on them. Finally, connect $t_1$ and $t_2$ with a curve through $R_2$ and put a crossbar on it. Observe that, this ends the play as no more moves can be made and a total of $6$ moves are made. Thus, we are done with the second part of the proof as well. 
\end{proof}

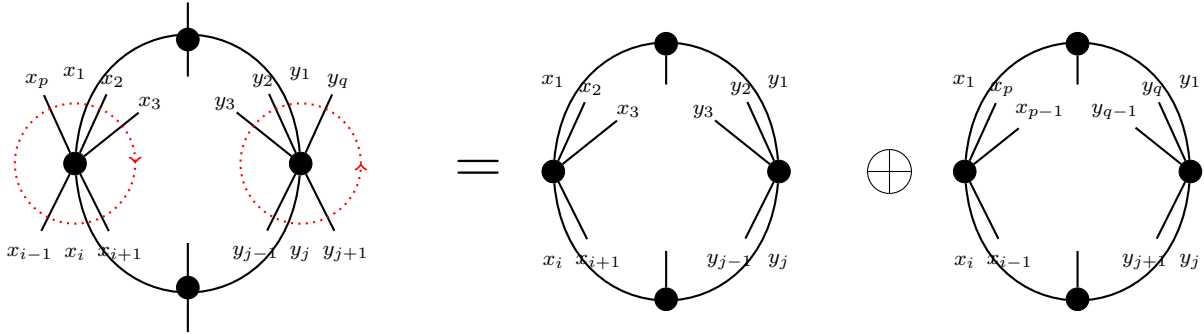
\begin{figure}%
    \centering
\begin{tikzpicture} 
\node[circle, draw,inner sep=0.1mm,minimum size=3mm, fill] (a) at (0,0) {};
\node [circle, draw, inner sep=0.1mm, minimum size=3mm, fill] (b) at (3,0) {};
\node [circle, draw,inner sep=0.1mm,minimum size=1mm, white] (a1) at (0,1.2) {\textcolor{black}{\scriptsize$x_1$}};
\node [circle, draw,inner sep=0.1mm,minimum size=1mm, white] (a2) at (0,-1.2) {\textcolor{black}{\scriptsize$x_i$}};
\node [circle, draw,inner sep=0.1mm,minimum size=1mm, white] (a3) at (0.5,1.1) {\textcolor{black}{\scriptsize$x_2$}};
\node [circle, draw,inner sep=0.1mm,minimum size=1mm, white] (a4) at (1,0.8) {\textcolor{black}{\scriptsize$x_3$}};
\node [circle, draw,inner sep=0.1mm,minimum size=1mm, white] (a5) at (-0.5,1.1) {\textcolor{black}{\scriptsize$x_p$}};
\node [circle, draw,inner sep=0.1mm,minimum size=1mm, white] (a6) at (-0.6,-1.2) {\textcolor{black}{\scriptsize$x_{i-1}$}};
\node [circle, draw,inner sep=0.1mm,minimum size=1mm, white] (a7) at (0.6,-1.2) {\textcolor{black}{\scriptsize$x_{i+1}$}};

\node [circle, draw,inner sep=0.1mm,minimum size=1mm, white] (b1) at (3,1.2) {\textcolor{black}{\scriptsize$y_1$}};
\node [circle, draw,inner sep=0.1mm,minimum size=1mm, white] (b2) at (3,-1.2) {\textcolor{black}{\scriptsize$y_j$}};
\node [circle, draw,inner sep=0.1mm,minimum size=1mm, white] (b3) at (2.5,1.1) {\textcolor{black}{\scriptsize$y_2$}};
\node [circle, draw,inner sep=0.1mm,minimum size=1mm, white] (b4) at (2,0.8) {\textcolor{black}{\scriptsize$y_3$}};
\node [circle, draw,inner sep=0.1mm,minimum size=1mm, white] (b5) at (3.5,1.1) {\textcolor{black}{\scriptsize$y_q$}};
\node [circle, draw,inner sep=0.1mm,minimum size=1mm, white] (b6) at (2.4,-1.2) {\textcolor{black}{\scriptsize$y_{j-1}$}};
\node [circle, draw,inner sep=0.1mm,minimum size=1mm, white] (b7) at (3.6,-1.2) {\textcolor{black}{\scriptsize$y_{j+1}$}};
\node [circle, draw, inner sep=0.1mm, minimum size=1mm, white] (c) at (1.5,2.2) {};
\node [circle, draw,inner sep=0.1mm,minimum size=1mm, white] (c1) at (1.5,1.1) {};
\node [circle, draw, inner sep=0.1mm, minimum size=1mm, white] (c2) at (1.5,-2.3) {};
\node [circle, draw,inner sep=0.1mm,minimum size=1mm, white] (c3) at (1.5,-1) {};
\node[circle, draw,inner sep=0.1mm,minimum size=3mm, fill] (c4) at (1.5,1.65) {};
\node[circle, draw,inner sep=0.1mm,minimum size=3mm, fill] (c5) at (1.5,-1.65) {};
\draw[thick] (a)--(a3);
\draw [thick](a)--(a4);
\draw [thick](a)--(a5);
\draw[thick] (a)--(a6);
\draw [thick](a)--(a7);
\draw[thick] (b)--(b3);
\draw [thick](b)--(b4);
\draw[thick] (b)--(b5);
\draw[thick] (b)--(b6);
\draw [thick](b)--(b7);
\draw[thick] (c)--(c1);
\draw[thick] (c2)--(c3);
\path [thick][thick](a) edge [bend left=85,looseness=1.8] (b);
\path[thick][thick] (a) edge [bend right=85,looseness=1.8] (b);

\draw[dotted,red,<-,thick] (0.8,0) arc[start angle=0, end angle=360, radius=0.8];
\draw[dotted,red,->,thick] (3.8,0) arc[start angle=0, end angle=360, radius=0.8];

\end{tikzpicture}
\qquad
\begin{tikzpicture} 
\node at (-1,0) {\huge $=$};
\node[circle, draw,inner sep=0.1mm,minimum size=3mm, fill] (a) at (0,0) {};
\node [circle, draw, inner sep=0.1mm, minimum size=3mm, fill] (b) at (3,0) {};
\node [circle, draw,inner sep=0.1mm,minimum size=1mm, white] (a1) at (0,1.2) {\textcolor{black}{\scriptsize$x_1$}};
\node [circle, draw,inner sep=0.1mm,minimum size=1mm, white] (a2) at (0,-1.2) {\textcolor{black}{\scriptsize$x_i$}};
\node [circle, draw,inner sep=0.1mm,minimum size=1mm, white] (a3) at (0.5,1.1) {\textcolor{black}{\scriptsize$x_2$}};
\node [circle, draw,inner sep=0.1mm,minimum size=1mm, white] (a4) at (1,0.8) {\textcolor{black}{\scriptsize$x_3$}};
\node [circle, draw,inner sep=0.1mm,minimum size=1mm, white] (a7) at (0.6,-1.2) {\textcolor{black}{\scriptsize$x_{i+1}$}};

\node [circle, draw,inner sep=0.1mm,minimum size=1mm, white] (b1) at (3,1.2) {\textcolor{black}{\scriptsize$y_1$}};
\node [circle, draw,inner sep=0.1mm,minimum size=1mm, white] (b2) at (3,-1.2) {\textcolor{black}{\scriptsize$y_j$}};
\node [circle, draw,inner sep=0.1mm,minimum size=1mm, white] (b3) at (2.5,1.1) {\textcolor{black}{\scriptsize$y_2$}};
\node [circle, draw,inner sep=0.1mm,minimum size=1mm, white] (b2) at (2.35,-1.2) {\textcolor{black}{\scriptsize$y_{j-1}$}};
\node [circle, draw,inner sep=0.1mm,minimum size=1mm, white] (b4) at (2,0.8) {\textcolor{black}{\scriptsize$y_3$}};

\node[circle, draw,inner sep=0.1mm,minimum size=3mm, fill] (c4) at (1.5,1.7) {};
\node[circle, draw,inner sep=0.1mm,minimum size=3mm, fill] (c5) at (1.5,-1.7) {};
\draw[thick][thick] (a)--(a3);
\draw [thick][thick](a)--(a4);
\draw[thick][thick] (a)--(a7);
\draw[thick][thick] (b)--(b3);
\draw [thick][thick](b)--(b4);
\draw[thick][thick] (b)--(b6);
\draw [thick][thick](c4)--(c1);
\draw [thick][thick](c5)--(c3);
\path [thick][thick](a) edge [bend left=85,looseness=1.8] (b);
\path[thick][thick] (a) edge [bend right=85,looseness=1.8] (b);


\end{tikzpicture}
\qquad
\begin{tikzpicture} 
\node at (-1,0) (XOR-aa)[XOR,scale=1.5] {};
\node[circle, draw,inner sep=0.1mm,minimum size=3mm, fill] (a) at (0,0) {};
\node [circle, draw, inner sep=0.1mm, minimum size=3mm, fill] (b) at (3,0) {};
\node [circle, draw,inner sep=0.1mm,minimum size=1mm, white] (a1) at (0,1.2) {\textcolor{black}{\scriptsize$x_1$}};
\node [circle, draw,inner sep=0.1mm,minimum size=1mm, white] (a2) at (0,-1.2) {\textcolor{black}{\scriptsize$x_i$}};
\node [circle, draw,inner sep=0.1mm,minimum size=1mm, white] (a3) at (0.5,1.1) {\textcolor{black}{\scriptsize$x_p$}};
\node [circle, draw,inner sep=0.1mm,minimum size=1mm, white] (a4) at (1,0.8) {\textcolor{black}{\scriptsize$x_{p-1}$}};
\node [circle, draw,inner sep=0.1mm,minimum size=1mm, white] (a7) at (0.6,-1.2) {\textcolor{black}{\scriptsize$x_{i-1}$}};

\node [circle, draw,inner sep=0.1mm,minimum size=1mm, white] (b1) at (3,1.2) {\textcolor{black}{\scriptsize$y_1$}};
\node [circle, draw,inner sep=0.1mm,minimum size=1mm, white] (b2) at (3,-1.2) {\textcolor{black}{\scriptsize$y_j$}};
\node [circle, draw,inner sep=0.1mm,minimum size=1mm, white] (b3) at (2.5,1.1) {\textcolor{black}{\scriptsize$y_q$}};
\node [circle, draw,inner sep=0.1mm,minimum size=1mm, white] (b4) at (2,0.8) {\textcolor{black}{\scriptsize$y_{q-1}$}};
\node [circle, draw,inner sep=0.1mm,minimum size=1mm, white] (b6) at (2.4,-1.2) {\textcolor{black}{\scriptsize$y_{j+1}$}};

\node[circle, draw,inner sep=0.1mm,minimum size=3mm, fill] (c4) at (1.5,1.7) {};
\node[circle, draw,inner sep=0.1mm,minimum size=3mm, fill] (c5) at (1.5,-1.7) {};
\draw[thick] (a)--(a3);
\draw [thick](a)--(a4);
\draw [thick](a)--(a7);
\draw[thick] (b)--(b3);
\draw[thick] (b)--(b4);
\draw[thick] (b)--(b6);
\draw [thick][thick](c4)--(c1);
\draw [thick][thick](c5)--(c3);
\path[thick] (a) edge [bend left=85,looseness=1.8] (b);
\path[thick] (a) edge [bend right=85,looseness=1.8] (b);


\end{tikzpicture}

    \caption{Expressing a partially played {\sc Brussels sprout} game as a sum of two {\sc Circular sprout} games.} \label{p+q_figures}
\end{figure}

From the above theorem, we can notice that the game 
$BS_2(\mathcal{P}_4:p,q)$ does not have a clear winner, and depending on the play, either Player 1 or Player 2 can win. Therefore, studying which player has a winning strategy makes sense. However, when we tried to do it by hand, it became extremely difficult, even for small values of $p$ and $q$.  Thus, a natural question to ask in this context is the following.

\begin{question}\label{BS_2(P_4)}
Which player has a winning strategy for the game 
$BS_2(\mathcal{P}_4:p,q)$ for all $p,q \geq 3$?
\end{question}

Keeping in mind, that our main goal of this article is to develop the theory around Sprout's conjecture (that is, Conjecture~\cite{lemoine2010computer}, we will, in fact, consider an even more difficult problem and solve it to get a conclusive answer to Question~\ref{BS_2(P_4)} as a direct corollary in the upcoming sections. The above mentioned stronger result requires introducing a new game called the {\sc Circular sprout} and involves the notion of ``nimber''~\cite{albert2019lessons}.





\section{The {\sc Circular sprout} game and its relation with $BS_2(\mathcal{P}_4:p,q)$}\label{sec circular}
While studying the game $BS_2(\mathcal{P}_4:p,q)$, we encountered another similar game which we found to be  interesting. 
Let us define this new game independently, and in a generalized form, even though in this article we will study only a specific restriction of it which will help us in 
studying $BS_2(\mathcal{P}_4:p,q)$. 

This new game, named the $n$-{\sc Circular sprout} game for the family $\mathcal{F}$ with parameters $(t_1, t_2, \cdots, t_n)$, is denoted by the notation  $CS_n(\mathcal{F}: t_1, t_2, \cdots, t_n)$. 
The initial set up of this game consists of $n$ dots $v_1, v_2, \cdots, v_n$  arranged in a clockwise order on the perimeter of a circle with $v_i$ having $t_i$ open tips coming out in the interior of the circle. The rest of the rules of the game are the same as {\sc Brussels sprout} with the following added constraint: the curves drawn by the players must be entirely contained in the interior of the circle.

Next let us observe how this game is related to $BS_2(\mathcal{P}_4:p,q)$. Let the open tips around the first spot be $x_1, x_2, \cdots, x_p$ arranged in a clockwise order, and let the open tips around the second spot be 
$y_1, y_2, \cdots, y_q$ arranged in an anti-clockwise order. 
Observe that, the very first move by Player 1 in the game $BS_2(\mathcal{P}_4:p,q)$ is 
unique up to renaming of the open tips. Therefore, without loss of generality one can assume that the very first move is joining the open tip $x_1$ with the open tip $y_1$ with a curve and then placing a crossbar on it. 
After that, the second player is forced to join an open tip $x_i$ to an open tip $y_j$, for some $i,j \neq 1$. This move will enable us to write the present game as the sum of two {\sc Circular sprout} game: $CS_4(\mathcal{P}_4: i-2, 1, j-2, 1)$ and 
$CS_4(\mathcal{P}_4: p-i, 1, q-j, 1)$ for some $i \in \{2, 3, \cdots, p\}$ and $j \in \{2, 3, \cdots, q\}$. Hence it will be enough to study and understand 
the games of the type $CS_4(\mathcal{P}_4: p, 1, q, 1)$ (see Figure~\ref{p+q_figures} for reference).

\subsection{Nimber}
Before moving forward with the study of these games, we would like to point out that
all the games discussed here are two player finite impartial games, and thus their  \textit{nimber}~\cite{siegel2013combinatorial}  can be calculated~\cite{siegel2013combinatorial}. To conclude which player has a winning strategy for a particular two player impartial game, it is enough to calculate the nimber value of the game; the second player has an winning strategy if and only if the nimber of a game is $0$~\cite{albert2019lessons}.

Recall that, to calculate the nimber of a game $X$, one first need to generate 
the entire game tree having $X$ as its root. 
Next, the leaves of the tree are all assigned nimber equal to $0$, while 
for the other nodes, its nimber is the
least non-negative integer which does not
occur as a nimber of any of its children. Let us denote the nimber of a game $X$ by $\eta[X]$ for convenience. 
We know that if an impartial game 
$X$ can be written as a sum of two impartial games, $Y$ and $Z$, then the nimber of $X$ can be given by $\eta[X] = \eta[Y] \oplus \eta[Z]$, where $\oplus$ denote the XOR operation~\cite{albert2019lessons}.

Thus our objective now is to calculate nimber of the game 
$CS_4(\mathcal{P}_4: p, 1, q, 1)$ for $p,q \geq 0$. Note that, the games 
$CS_4(\mathcal{P}_4: p, 1, q, 1)$ and 
$CS_4(\mathcal{P}_4: q, 1, p, 1)$ are the exact same games up to symmetry. Thus it is enough to calculate 
$\eta[CS_4(\mathcal{P}_4: p, 1, q, 1)]$ for all $q \geq p \geq 0$.

\section{Nimber of $CS_4(\mathcal{P}_4:p,1,q,1)$}\label{sec nimber of circular}\label{Nimbersec}
We assume $CS_4(\mathcal{P}_4:p,1,q,1)$ to be embedded as $p$ on the left, $q$ on the right and the singular open tips to be at the top and bottom.


\begin{theorem}\label{th nimber}
For all $0 \leq p \leq q$, we have 
\begin{equation*}
\eta[CS_4(\mathcal{P}_4:p, 1, q, 1)] =
\begin{cases} 
      1 & \text{if $p=q$}, \\
      \frac{4}{5} (p+q-i)+2 \lfloor\frac{i}{4}\rfloor & \text{if $p<q<2p-\lfloor\frac{|p-2|}{2}\rfloor$ where $i \equiv p+q (mod\ 5)$}\\
      2p & \text{if $q \geq 2p-\lfloor\frac{|p-2|}{2}\rfloor$.}
   \end{cases}
\end{equation*}
\label{eq formula t1t3}
where $1,2,3,4,5$ are the representative of the integers modulo $5$.
\end{theorem}

\subsection{Proof of Theorem~\ref{th nimber}}\label{sec proofs}
The proof of Theorem~\ref{th nimber} uses the method of structural induction and is lengthy and complicated. For the benefit of the reader, we will divide the proof into several lemmas and observations, and also present it across different subsections.

\subsubsection{The basic set up for the proof}
Let us assume a typical pictorial representation of the game $CS_4(\mathcal{P}_4:p, 1, q, 1)$ to be as follows.
Let the reference circle of the game 
be the circle having unit radius and $(0,0)$ as its centre. 
After that, let the four spots on it, along with their positions, be
$x$ at $(-1,0)$, $t$ at $(0,1)$, $y$ at $(1,0)$ and $b$ at $(0,-1)$. 
Furthermore, 
let $x_1, x_2, ..., x_p$ be the open
tips coming out of $x$, arranged in a clockwise order around $x$ and 
let $y_1, y_2, ..., y_q$ be the open tips coming out of $y$, arranged in an 
anti-clockwise order around $y$. Also let $t_1$ and $b_1$ be the open tips coming out of $t$ and $b$, respectively.

Notice that, there can be two types of moves that the first player can perform on the given initial stage of the game $CS_4(\mathcal{P}_4:p, 1, q, 1)$. 
The first type is to connect $t_1$ with $b_1$ with a curve and then put a crossbar on it. However, this creates a game equivalent to the sum of the two games
$CS_4(\mathcal{P}_4:p, 0, 1, 0)$ and $CS_4(\mathcal{P}_4:1, 0, q, 0)$. 
The second type of move 
is to join $x_{p'+1}$ with $y_{q'+1}$ with a curve and then put a crossbar on it. This 
creates a game equivalent to the sum of the two games 
$CS_4(\mathcal{P}_4:p', 1, q', 1)$ and $CS_4(\mathcal{P}_4: p'', 1, q'', 1)$, where 
$p' \in \{0,1,\cdots, p'-1\}$,
$q' \in \{0,1,\cdots, q'-1\}$, 
$p'' = p-p'-1$ and $q'' = q-q'-1$. 
From now on, we will assume 
$p' \in \{0,1,\cdots, p'-1\}$,
$q' \in \{0,1,\cdots, q'-1\}$, 
$p'' = p-p'-1$ and $q'' = q-q'-1$ as general convention for this proof.

Observe that, irrespective of the moves, the game becomes a sum of similar kinds of games allowing us to use induction. That is why, we will use the method of strong induction on 
$p$ to prove this result.

\subsubsection{Base case}
Before going to the base case, let us first prove some useful lemmas. 

\begin{lemma}
We have
\begin{enumerate}[(i)]
    \item $\eta[CS_4(\mathcal{P}_4:0, 0, q, 0)] = 0$ for all $q \geq 0$, 
    \item $\eta[CS_4(\mathcal{P}_4:0, 1, 0, 1)] = 1$,
    
    \item $\eta[CS_4(\mathcal{P}_4:1, 0, q, 0)]=1$ for all $q \geq 1$, 
    
     \item $\eta[CS_4(\mathcal{P}_4:0, 1, q, 1)]=0$ for all  $q\geq 1$. 
\end{enumerate}
\end{lemma}

\begin{proof}
(i) Notice that the game $CS_4(\mathcal{P}_4:0, 0, q, 0)$ does not have any move, and thus trivially, 
\begin{equation}\label{eq nim 00q0}
    \eta[CS_4(\mathcal{P}_4:0, 0, q, 0)] = 0.
\end{equation}

\medskip

\noindent (ii) Observe that the game $CS_4(\mathcal{P}_4:0, 1, 0, 1)$ has exactly one move, that is, connecting $t_1$ and $b_1$ with a curve, and this move leaves the game in a position where no further moves are possible. Thus,  
\begin{equation}\label{eq nim 0101}
    \eta[CS_4(\mathcal{P}_4:0, 1, 0, 1)] = 1.
\end{equation}

\noindent (iii) Note  that the only possible move in the game $CS_4(\mathcal{P}_4:1, 0, q, 0)$ 
is to connect $x_1$ 
to some $y_{q'+1}$ after which the game will end.
Therefore, 
\begin{equation}\label{eq nim 10q0}
    \eta[CS_4(\mathcal{P}_4:1, 0, q, 0)]=1.
\end{equation}

\noindent (iv) In the game $CS_4(\mathcal{P}_4:0, 1, q, 1)$ for $q \geq 1$, the only possible first move is 
to connect $t_1$ with $b_1$ by a curve. This can be expressed as a sum of the games 
$CS_4(\mathcal{P}_4: 0, 0, 1, 0)$  and $CS_4(\mathcal{P}_4:1, 0, q, 0)$. We already know that 
$\eta[CS_4(\mathcal{P}_4:0, 0, q, 0)] = 0$ by equation~(\ref{eq nim 00q0}), and $\eta[CS_4(\mathcal{P}_4:1, 0, q, 0)]=1$ due to equation~\ref{eq nim 10q0}. 
Hence the only child of the game $CS_4(\mathcal{P}_4:0, 1, q, 1)$
has nimber 
\begin{equation}\label{eq nim 00q0+10q0}
    \eta[CS_4(\mathcal{P}_4:0, 0, q, 0)] \oplus \eta[CS_4(\mathcal{P}_4:1, 0, q, 0)]=1
\end{equation}
which implies 
\begin{equation}\label{eq nim 01q1}
   \eta[CS_4(\mathcal{P}_4:0, 1, q, 1)]=0.
\end{equation}

\medskip 

This completes the proof. 
\end{proof}

 This implies that the formula given in the statement of Theorem~\ref{th nimber} holds when $p=0$. 

\begin{lemma}
We have 
\begin{equation}
  \eta[CS_4(\mathcal{P}_4:1, 1, q, 1)]  = 
   \begin{cases}
   1 &\text{ if } q = 1,\\
   2 &\text{ if } q \geq 2.
    \end{cases}
\end{equation}
\end{lemma}

\begin{proof}
The game $CS_4(\mathcal{P}_4:1, 1, q, 1)$ has two types of children in the game tree. The first, obtained by connecting $t_1$ and $b_1$ with a curve, 
can be expressed as the sum of $CS_4(\mathcal{P}_4:1, 0, 1, 0)$ and 
$CS_4(\mathcal{P}_4:1, 0, q, 0)$. This game has nimber 
\begin{equation}\label{eq nim 1010+10q0}
    \eta[CS_4(\mathcal{P}_4:1, 0, 1, 0)] \oplus \eta[CS_4(\mathcal{P}_4:1, 0, q, 0)]=1 \oplus 1 = 0
\end{equation}
due to equation~(\ref{eq nim 0101}) and~(\ref{eq nim 10q0}).
The second child is obtained by connecting 
$x_1$ and $y_{q'+1}$, and thus can be expressed as a sum of  
$CS_4(\mathcal{P}_4:0, 1, q', 1)$ and $CS_4(\mathcal{P}_4:0, 1, q'', 1)$, and thus has 
nimber 
\begin{equation}\label{eq nim 01q'1+01q''1}
    \eta[CS_4(\mathcal{P}_4:0, 1, q', 1)] \oplus \eta[CS_4(\mathcal{P}_4:0, 1, q'', 1)]=0 \oplus 0 = 0
\end{equation}
for $q', q'' \geq 1$ due to equation~(\ref{eq nim 01q1}).
On the other hand, using equation~(\ref{eq nim 0101}) and~(\ref{eq nim 01q1}), 
when $q'=0$ (or $q'' = 0$) we have 
\begin{equation}\label{eq nim 0101+01(q-1)1}
    \eta[CS_4(\mathcal{P}_4:0, 1, 0, 1)] \oplus \eta[CS_4(\mathcal{P}_4:0, 1, q-1, 1)]=1 \oplus 0 = 1
\end{equation}
for $q \geq 2$. Notice that, if $q = 1$, then 
\begin{equation}\label{eq nim 0101+0101}
    \eta[CS_4(\mathcal{P}_4:0, 1, 0, 1)] \oplus \eta[CS_4(\mathcal{P}_4:0, 1, 0, 1)]=1 \oplus 1 = 0.
\end{equation}
That means, the game $CS_4(\mathcal{P}_4:1, 1, q, 1)$ has children with nimber $0$ 
and $1$ when $q \geq 2$. 
On the other hand, if $q=1$, 
it has only children with nimber $0$. Therefore, we have 
\begin{equation}
  \eta[CS_4(\mathcal{P}_4:1, 1, q, 1)]  = 
   \begin{cases}
   1 &\text{ if } q = 1,\\
   2 &\text{ if } q \geq 2.
    \end{cases}
\end{equation}

This completes the proof. 
\end{proof}

Thus, the formula given in the statement is true for $p=1$ as well. 
Hence we have verified the formula to be true for $p \leq 1$, and this is the base case of our induction.

\subsubsection{Induction hypothesis}
Let us assume that the formula is true for all   
$\eta[CS_4(\mathcal{P}_4:p-i, 1, q, 1)]$, where $i \in \{1, 2, \cdots, p\}$. For the induction step, we need to show that it is true for $\eta[CS_4(\mathcal{P}_4:p, 1, q, 1)]$ as well.

\subsubsection{Induction step: A key lemma}
We will break our induction step of the proof across several lemmas. 
However, before that we will make a general observation useful for all cases. 

\begin{lemma}\label{lem p1q1 geq 1}
For all $0 \leq p \leq q$,
\begin{equation}\label{eq p1q1 geq 1}
    \eta[CS_4(\mathcal{P}_4:p, 1, q, 1)] \geq 1. 
\end{equation}
\end{lemma}

\begin{proof}
Notice that each of the games $CS_4(\mathcal{P}_4:p, 1, q, 1)$ has one child which is expressed as a sum of $CS_4(\mathcal{P}_4:p, 0, 1, 0)$ and $CS_4(\mathcal{P}_4:1, 0, q, 0)$. By equation~\ref{eq nim 10q0}, this game
has nimber 
\begin{equation}\label{eq nim 0101+0101}
    \eta[CS_4(\mathcal{P}_4:p,0,1,0)] \oplus \eta[CS_4(\mathcal{P}_4:1,0,q,0)]=1 \oplus 1 = 0.
\end{equation}
Therefore, all the games of the type $CS_4(\mathcal{P}_4:p, 1, q, 1)$ has a child with nimber $0$. This implies 
\begin{equation}\label{eq p1q1 geq 1}
    \eta[CS_4(\mathcal{P}_4:p, 1, q, 1)] \geq 1. 
\end{equation}
Hence the proof. 
\end{proof}

\subsubsection{Induction step: When $p=q$.}
Now we are ready to handle the case where $p = q$.

\begin{lemma}
If $\eta[CS_4(\mathcal{P}_4:n, 1, q, 1)]$ satisfy the formula 
given in the statement of Theorem~\ref{th nimber} for all 
$n \leq p-1$, then $\eta[CS_4(\mathcal{P}_4:p, 1, p, 1)] = 1$.
\end{lemma}

\begin{proof}
 When $p = q$, note that because of Lemma~\ref{lem p1q1 geq 1} 
 it is enough to show that none of the children of $CS_4(\mathcal{P}_4:p, 1, q, 1)$ in the game tree has nimber equal to $1$. 
That is, we need to show that 
$$\eta[CS_4(\mathcal{P}_4:p', 1, q', 1)] \oplus \eta[CS_4(\mathcal{P}_4: p'', 1, q'', 1)] \neq 1.$$
Notice that, by our induction hypothesis, 
both $\eta[CS_4(\mathcal{P}_4:p', 1, q', 1)]$
and $\eta[CS_4(\mathcal{P}_4: p'', 1, q'', 1)]$ has even values unless $p'=q'$, and hence, as the XOR of two even numbers is even, we are done unless $p'=q'$. 
If $p' = q'$, then both 
$\eta[CS_4(\mathcal{P}_4:p', 1, q', 1)]$ 
and $\eta[CS_4(\mathcal{P}_4: p'', 1, q'', 1)]$ are equal to $1$. Thus, in this case also, 
$$\eta[CS_4(\mathcal{P}_4:p', 1, p', 1)] \oplus \eta[CS_4(\mathcal{P}_4: p'', 1, p'', 1)] = 1 \oplus 1 = 0.$$
This implies that 
$$\eta[CS_4(\mathcal{P}_4:p, 1, p, 1)] = 1$$
and concludes the proof.
\end{proof}

\subsubsection{Induction step: When $q \geq 2p - \lfloor \frac{p-2}{2} \rfloor$.}
\begin{lemma}
If $\eta[CS_4(\mathcal{P}_4:n, 1, q, 1)]$ satisfies the formula 
given in the statement of Theorem~\ref{th nimber} for all 
$n \leq p-1$, then $\eta[CS_4(\mathcal{P}_4:p, 1, q, 1)] = 2p$
when $q \geq 2p - \lfloor \frac{p-2}{2} \rfloor$.
\end{lemma}

\begin{proof}  When $q \geq 2p - \lfloor \frac{p-2}{2} \rfloor$, we need to show that none of the children of $CS_4(\mathcal{P}_4:p, 1, q, 1)$ in the game tree 
has nimber equal to $2p$, where as, for each $i \in \{1,2, \cdots, 2p-1\}$, there exists a child of $CS_4(\mathcal{P}_4:p, 1, q, 1)$ in the game tree which has nimber equal to $i$. 

First, we will show that the odd numbers less than $2p$ appear as the nimbers of the
children of $CS_4(\mathcal{P}_4:p, 1, q, 1)$ in the game tree. 
For that, let us consider the children that can be expressed as sum of the games 
$CS_4(\mathcal{P}_4:p', 1, q', 1)$ and $CS_4(\mathcal{P}_4: p'', 1, q'', 1)$, 
where $q \geq 2p - \lfloor \frac{p-2}{2} \rfloor$ and $p' = q'$. 
In this scenario, note that we have 
$$\eta[CS_4(\mathcal{P}_4:p', 1, q', 1)] = \eta[CS_4(\mathcal{P}_4:p', 1, p', 1)] = 1$$
by induction hypothesis. 
Moreover, observe that 
\begin{align*}
     q'' = q-1-q' &\geq \left(2p - \lfloor \frac{p-2}{2} \rfloor \right) - 1 - p' \\
     &\geq  \left(2p - \lfloor \frac{p-2}{2} \rfloor -1 -p'\right) + (p' - p') + \left(\lfloor \frac{p''-2}{2} \rfloor - \lfloor \frac{p''-2}{2}\rfloor \right ) + (1 - 1)\\
 &\geq \left(2(p-p'-1) - \lfloor \frac{p''-2}{2} \rfloor\right) + \left(p' + \lfloor \frac{p''-2}{2} \rfloor - \lfloor \frac{p-2}{2} \rfloor +1 \right)\\
&\geq 2p'' - \lfloor \frac{p''-2}{2}  \rfloor.
\end{align*}
This implies 
$$\eta[CS_4(\mathcal{P}_4: p'', 1, q'', 1)] = 2p''$$ due to the induction hypothesis. 
Therefore, 
$$\eta[CS_4(\mathcal{P}_4:p', 1, p', 1)] \oplus \eta[CS_4(\mathcal{P}_4: p'', 1, q'', 1)] = 1 \oplus 2p'' =2p''+1$$ 
is the nimber of a child of $CS_4(\mathcal{P}_4:p, 1, q, 1)$ in the game tree. 
As $p'' \in \{0, 1, \cdots, p-1\}$, the above covers all odd numbers less than $2p$. 

Next, we will show that the even numbers less than $2p$ appear as the nimbers of the
children of $CS_4(\mathcal{P}_4:p, 1, q, 1)$ in the game tree. 
For that, let us consider the children that can be expressed as the sum of the games 
$CS_4(\mathcal{P}_4:p', 1, q-1, 1)$ and $CS_4(\mathcal{P}_4: p'', 1, 0, 1)$.
In this scenario, note that we have 
$$\eta[CS_4(\mathcal{P}_4: p'', 1, 0, 1)] = 0$$
by equation~(\ref{eq nim 01q1}). 
Also 
\begin{align*}
     q-1  &\geq (2p - \lfloor \frac{p-2}{2} \rfloor) - 1 \\ 
     &\geq  \left(2p - \lfloor \frac{p-2}{2} \rfloor - 1 \right) + (2p' - 2p') + \left(\lfloor \frac{p'-2}{2} \rfloor - \lfloor \frac{p'-2}{2} \rfloor \right)\\
 &\geq  (2p' - \lfloor \frac{p'-2}{2} \rfloor) + \left(2p - 2p' + \lfloor \frac{p'-2}{2} \rfloor - \lfloor \frac{p-2}{2} \rfloor -1 \right)\\
&\geq 2p' - \lfloor \frac{p'-2}{2}  \rfloor.
\end{align*}
This implies 
$$\eta[CS_4(\mathcal{P}_4:p', 1, q-1, 1)] = 2p'$$ due to the induction hypothesis. 
Therefore, 
$$\eta[CS_4(\mathcal{P}_4:p', 1, q-1, 1)] \oplus \eta[CS_4(\mathcal{P}_4: p'', 1, 0, 1)] = 2p' \oplus 0 = 2p'$$
is the nimber of a child of $CS_4(\mathcal{P}_4:p, 1, q, 1)$ in the game tree. 
As $p' \in \{0, 1, \cdots, p-1\}$, the above covers all even numbers less than $2p$.

Now let us show that there is no child of $CS_4(\mathcal{P}_4:p, 1, q, 1)$ in the game tree whose nimber equals $2p$. Notice that, 
in general 
$$\eta[CS_4(\mathcal{P}_4:p', 1, q', 1)] \leq 2p'$$ due to the induction hypothesis. 
Therefore, 
$$\eta[CS_4(\mathcal{P}_4:p', 1, q', 1)] \oplus \eta[CS_4(\mathcal{P}_4: p'', 1, q'', 1)] \leq 2p' + 2p'' = 2(p' + p-p'-1)  < 2p.$$
Hence we are done. 
\end{proof}

\subsubsection{Induction step: When $p < q < 2p - \lfloor \frac{p-2}{2} \rfloor$.} 
  When $p < q < 2p - \lfloor \frac{p-2}{2} \rfloor$, 
  proving the induction step needs careful treatment. 
  Due to equation~(\ref{eq p1q1 geq 1}), we know that 
$$\eta[CS_4(\mathcal{P}_4:p, 1, q, 1)] \geq 1.$$
Now to prove this case, first we need to show that 
$$\eta[CS_4(\mathcal{P}_4:p', 1, q', 1)] \oplus \eta[CS_4(\mathcal{P}_4: p'', 1, q'', 1)] \neq \frac{4}{5} (p+q-i)+2 \lfloor\frac{i}{4}\rfloor,$$ 
where $i \equiv p+q (mod\ 5)$   
and  $1,2,3,4,5$ are the representative of the integers modulo $5$.

\begin{lemma}
If $\eta[CS_4(\mathcal{P}_4:n, 1, q, 1)]$ satisfies the formula 
given in the statement of Theorem~\ref{th nimber} for all 
$n \leq p-1$, then 
$$\eta[CS_4(\mathcal{P}_4:p', 1, q', 1)] \oplus \eta[CS_4(\mathcal{P}_4: p'', 1, q'', 1)] \neq 
\frac{4}{5} (p+q-i)+2 \lfloor\frac{i}{4}\rfloor,$$ 
where $p < q < 2p - \lfloor \frac{p-2}{2} \rfloor$, 
$i \equiv p+q (mod\ 5)$,   
and  $1,2,3,4,5$ are the representative of the integers modulo $5$. 
\end{lemma}

\begin{proof} 
Without loss of generality, let us assume $q' \geq p'$. 
When $p'=q'$, we have 
$$\eta[CS_4(\mathcal{P}_4:p', 1, q', 1)] \oplus \eta[CS_4(\mathcal{P}_4: p'', 1, q'', 1)] \neq \frac{4}{5} (p+q-i)+2 \lfloor\frac{i}{4}\rfloor$$
as the left-hand side is of the form  $1 \oplus x$, where $x$ is an even number. That is, the left-hand side is an odd number while the right-hand side is an even number.

Next suppose that $p' < q' < 2p' - \lfloor \frac{p'-2}{2} \rfloor$ and 
$p'' < q'' < 2p'' - \lfloor \frac{p''-2}{2} \rfloor$ 
(or $q'' < p'' < 2q'' - \lfloor \frac{q''-2}{2} \rfloor$). Moreover, let 
$i \equiv (p+q) \pmod 5$, $i' \equiv (p'+q') \pmod 5$, and 
$i'' \equiv (p''+q'') \pmod 5$. In this scenario, $i'$ and $i''$
can have some specific values depending on $i$. If we know these values, then it is possible to compute and compare the terms and show that 
$$\eta[CS_4(\mathcal{P}_4:p', 1, q', 1)] \oplus \eta[CS_4(\mathcal{P}_4: p'', 1, q'', 1)] < \frac{4}{5} (p+q-i)+2 \lfloor\frac{i}{4}\rfloor.$$
We will demonstrate a sample calculation for one of the cases, and omit the detailed calculations for the others. 
However, we will summarize them in a consolidated manner in the comparison table (Table~\ref{tab all cases}). 

\medskip

\noindent \textit{A sample calculation:} If $i = 4$, the possible values for $\{i', i''\}$ are $\{5,2\}$, $\{1,1\}$, and $\{3,4\}$. 
When  $\{i', i''\} = \{5,2\}$ we have
\begin{align*}
    \eta[CS_4(\mathcal{P}_4:p', 1, q', 1)] \oplus \eta[CS_4(\mathcal{P}_4: p'', 1, q'', 1)] &= \frac{4}{5} (p+q - 2 - 5 - 2) + 2 \\
    &= \frac{4}{5} (p+q+1) - 6 \\
    &\neq \frac{4}{5} (p+q+1) - 2.
\end{align*}


\begin{table}[ht]
    \centering
   \begin{tabular}{ |c|c|c|c| } 

 \hline
 
  &  &  & \\ 
  $i$ & $\{i',i''\}$ & $\eta[CS_4(\mathcal{P}_4:p', 1, q', 1)] \oplus \eta[CS_4(\mathcal{P}_4:p'', 1, q'', 1)]$ & $\frac{4}{5} (p+q-i)+2 \lfloor\frac{i}{4}\rfloor$\\ 
    &  & & \\ 
 
 \hline

  & $\{5,3\}$ & $\frac{4}{5}(p+q)-6$ & \\ 
 $5$ & $\{1,2\}$ & $\frac{4}{5}(p+q)-4$ & $\frac{4}{5}(p+q)-2$\\
 & $\{4,4\}$ & $\frac{4}{5}(p+q)-4$ & \\
 
 \hline
 
  & $\{5,2\}$ & $\frac{4}{5}(p+q+1)-6$ & \\
 $4$ & $\{1,1\}$ & $\frac{4}{5}(p+q+1)-4$ & $\frac{4}{5}(p+q+1)-2$ \\ 
  & $\{3,4\}$ & $\frac{4}{5}(p+q+1)-6$ & \\
  
 \hline
 & $\{5,1\}$ & $\frac{4}{5}(p+q+2)-6$ & \\
 $3$ & $\{2,4\}$ & $\frac{4}{5}(p+q+2)-6$ & $\frac{4}{5}(p+q+2)-4$\\ 
  & $\{3,3\}$ & $\frac{4}{5}(p+q+2)-8$ & \\
  
\hline

 & $\{5,5\}$ & $\frac{4}{5}(p+q+3)-8$ & \\
 $2$ & $\{4,1\}$ & $\frac{4}{5}(p+q+3)-6$ & $\frac{4}{5}(p+q+3)-4$ \\ 
  & $\{3,2\}$ & $\frac{4}{5}(p+q+3)-8$ & \\
  
 \hline
 
 & $\{5,4\}$ & $\frac{4}{5}(p+q+4)-8$ & \\
 $1$ & $\{3,1\}$ & $\frac{4}{5}(p+q+4)-8$ &$\frac{4}{5}(p+q+4)-4$ \\ 
  & $\{2,2\}$ & $\frac{4}{5}(p+q+4)-8$ & \\
  
 \hline

\end{tabular}
    \caption{The comparison table.}\label{tab all cases}
    \label{tab:my_label}
\end{table}

\medskip

When $ q' \geq 2p' - |\lfloor \frac{p'-2}{2} \rfloor|$, then note that 
\begin{equation}\label{eq balanced unbalanced}
\begin{split}
    q'' = q-q'-1 &\leq q -  2p' + |\lfloor \frac{p'-2}{2}  \rfloor| - 1\\              &\leq q - 2p' + \frac{p'}{2} - 2\\
             &\leq  q - 1.5(p-p''-1)  -2 \\
             &\leq  (q-1.5p) +1.5p'' - 0.5\\
             &< 1.5p'' -1  \\
             &\leq 2p'' - \lfloor \frac{p''-2}{2} \rfloor. 
\end{split}
\end{equation}
The above inequalities follow if we 
have $p' \geq 2$. Otherwise also, for the specific case of $p'=0$ or $1$, one can still prove  $q'' \leq 2p'' - \lfloor \frac{|p''-2|}{2} \rfloor$ directly.

Thus, if $p''+q'' \equiv i'' \pmod 5$, we must have 
\begin{align*}
    \eta[CS_4(\mathcal{P}_4:p', 1, q', 1)] &\oplus \eta[CS_4(\mathcal{P}_4: p'', 1, q'', 1)] \\
&\leq 2p' + \frac{4}{5}(p''+q''-i'') + 2 \lfloor \frac{i''}{4} \rfloor \\
&= 2p' + \frac{4}{5}((p - p' - 1)+ (q - q' - 1) -i'') + 2 \lfloor \frac{i''}{4} \rfloor \\
&= 2p' + \frac{4}{5}(p+q) - \frac{4}{5}((p'+q')+2+i'') + 2 \lfloor \frac{i''}{4} \rfloor \\
&\leq 2p' + \frac{4}{5}(p+q) - \frac{4}{5}((p'+2p'-\lfloor \frac{p'-2}{2} \rfloor)+2+i'') + 2 \lfloor \frac{i''}{4} \rfloor \\
&\leq 2p' + \frac{4}{5}(p+q) - \frac{4}{5}((p'+2p'- \frac{p'}{2}  + 1)+2+i'') + 2 \lfloor \frac{i''}{4} \rfloor \\
&\leq \frac{4}{5}(p+q) - \frac{4}{5}(3 +i'') + 2 \lfloor \frac{i''}{4} \rfloor \\
&< \frac{4}{5}(p+q -i) + 2 \lfloor \frac{i}{4} \rfloor \\
\end{align*}
where $p+q \equiv i \pmod 5$. 
Note that, in the last step, strict inequality holds irrespective of the values of $i$ and $i''$. 
Furthermore, the above inequalities follow if we 
have $p' \neq 1$. Otherwise also, for the specific case of $p'=1$, one can still prove that the left-hand side of the inequality is less than the 
eventual right-hand side of the inequality. 
\end{proof}

Now it remains to show that 
for each $x \in \{1,2, \cdots,  \frac{4}{5}(p+q -i) + 2 \lfloor \frac{i}{4} \rfloor\}$ where $(p+q) \equiv i \pmod 5$ it is possible to find $p'$ and $q'$ such that 
$$\eta[CS_4(\mathcal{P}_4:p', 1, q', 1)] \oplus \eta[CS_4(\mathcal{P}_4: p'', 1, q'', 1)] = x.$$

\begin{lemma}
Suppose that $\eta[CS_4(\mathcal{P}_4:n, 1, q, 1)]$ satisfies the formula given in the statement of Theorem~\ref{th nimber} for all 
$n \leq p-1$ where $p < q < 2p - \lfloor \frac{p-2}{2} \rfloor$.
Then there exists a child of  
the game $CS_4(\mathcal{P}_4:p, 1, q, 1)$ in its game tree having  nimber equal to $l$ where $l \in \{1, 3, \cdots, \frac{4}{5}(p+q -i) + 2 \lfloor \frac{i}{4} \rfloor - 1\}$ and  $(p+q) \equiv i \pmod 5$. 
\end{lemma}

\begin{proof}
Let $\frac{4}{5}(p+q -i) + 2 \lfloor \frac{i}{4} \rfloor = n_1$ 
and let $\eta[CS_4(\mathcal{P}_4:p-1, 1, q-1, 1)] = n_2$. 
We will show that $n_2 \geq n_1 - 2$. 

If $p-1 < q-1 < 2(p-1) -  \lfloor \frac{p-3}{2} \rfloor$, then, assuming $(p+q-2) \equiv i_2 \pmod 5$, we have  
\begin{align*}
    n_2 - n_1 &=  \left(\frac{4}{5}(p-1+q-1 -i_2) + 2 \lfloor \frac{i_2}{4} \rfloor\right)  - \left(\frac{4}{5}(p+q -i) + 2 \lfloor \frac{i}{4} \rfloor\right)\\
    &= \frac{4}{5}(i-i_2-2) + 2 \lfloor \frac{i_2}{4} \rfloor - 2 \lfloor \frac{i}{4} \rfloor.
\end{align*}
Notice that, $i = 3,4,5$ implies $i_2 = i-2$. Thus
$$n_2 - n_1 = 2 \lfloor \frac{i_2}{4} \rfloor - 2 \lfloor \frac{i}{4} \rfloor \geq - 2.$$
If $i=1,2$, then $i_2 = i+3$ which implies
$$n_2 - n_1 = -4  + 2 \lfloor \frac{i_2}{4} \rfloor - 2 \lfloor \frac{i}{4} \rfloor \geq -2 .$$
If $2(p-1) -  \lfloor \frac{p-3}{2} \rfloor \leq q-1 < 2p -  \lfloor \frac{p-2}{2} \rfloor$, then as the maximum difference between the lower and upper bound of $q-1$ is at most $1$, and as $q-1$ is an integer, we can say that 
$$q = 2(p-1) -  \lfloor \frac{p-3}{2} \rfloor +1.$$
Hence
\begin{align*}
    n_2 - n_1 &=  2(p-1)  - \frac{4}{5}(p+q -i) + 2 \lfloor \frac{i}{4} \rfloor\\
    &\geq 2p -2 - \frac{4}{5}(p + 2p -2- \lfloor \frac{p-3}{2} \rfloor +1 -i) + 2 \lfloor \frac{i}{4} \rfloor\\
    &\geq 2p -2 - \frac{4}{5}(3p -1 - (\frac{p-3}{2}) + 0.5  -i) + 2 \lfloor \frac{i}{4} \rfloor\\
    &= 2p-2-\frac{4}{5}(\frac{5p}{2}+1-i)+ 2 \lfloor \frac{i}{4} \rfloor\\
    &=-2-\frac{4}{5}(1-i)+ 2 \lfloor \frac{i}{4} \rfloor\\
    &\geq -2.
\end{align*}

Therefore, $n_2 \geq n_1 - 2$. That means, 
for each non-negative integer less than $n_2$
the game 
$CS_4(\mathcal{P}_4:p-1, 1, q-1, 1)$ 
has a child  having nimber equal to it. In particular, 
for each non-negative odd integer less than $n_2$, $CS_4(\mathcal{P}_4:p-1, 1, q-1, 1)$ 
has a child  having nimber equal to it.

Let $l \in \{1, 3, \cdots, n_1-3 \}$ be an odd integer. 
We know that, $CS_4(\mathcal{P}_4:p-1, 1, q-1, 1)$  has a child with nimber $l$. However, we also know that the child must be 
obtained by connecting $x_{a+1}$ and $y_{a+1}$ with a curve, and its nimber is given by
$$\eta[CS_4(\mathcal{P}_4:a, 1, a, 1)] \oplus \eta[CS_4(\mathcal{P}_4: p-a-2, 1, q-a-2, 1)] = l.$$
As $\eta[CS_4(\mathcal{P}_4:a, 1, a, 1)] = 1$, we must have 
$$\eta[CS_4(\mathcal{P}_4: p-a-2, 1, q-a-2, 1)] = l-1.$$
Thus the child of $CS_4(\mathcal{P}_4:p, 1, q, 1)$
obtained by connecting $x_{a+2}$ and $y_{a+2}$ has nimber 
$$\eta[CS_4(\mathcal{P}_4:a+1, 1, a+1, 1)] \oplus \eta[CS_4(\mathcal{P}_4: p-a-2, 1, q-a-2, 1)] = l.$$
Moreover, the
 child of $CS_4(\mathcal{P}_4:p, 1, q, 1)$
obtained by connecting $x_1$ and $y_{1}$ has nimber 
$$\eta[CS_4(\mathcal{P}_4:0, 1, 0, 1)] \oplus \eta[CS_4(\mathcal{P}_4: p-1, 1, q-1, 1)] = n_1-1.$$

Thus, for each $l \in \{1, 3, \cdots, n_1-1 \}$
there exists a child of $CS_4(\mathcal{P}_4:p, 1, q, 1)$ 
having nimber equal to $l$. 
\end{proof}

Last but not the least, we are going to show that the game 
$CS_4(\mathcal{P}_4:p, 1, q, 1)$  has a child with nimber $l$ for all even positive integers $l$ less than 
$\frac{4}{5}(p+q -i) + 2 \lfloor \frac{i}{4} \rfloor\}$ where  $(p+q) \equiv i \pmod 5$.

\begin{lemma}
Suppose that $\eta[CS_4(\mathcal{P}_4:n, 1, q, 1)]$ satisfies the formula given in the statement of Theorem~\ref{th nimber} for all 
$n \leq p-1$
where $p < q < 2p - \lfloor \frac{|p-2|}{2} \rfloor$.
Then there exists a child of  
the game $CS_4(\mathcal{P}_4:p, 1, q, 1)$ in its game tree having  nimber equal to $l$ where $l \in \{2, 4, \cdots, \frac{4}{5}(p+q -i) + 2 \lfloor \frac{i}{4} \rfloor - 2\}$ and  $(p+q) \equiv i \pmod 5$. 
\end{lemma}

\begin{proof}
Let $\frac{4}{5}(p+q -i) + 2 \lfloor \frac{i}{4} \rfloor = n_1$. We want to show that 
$$\eta[CS_4(\mathcal{P}_4:n, 1, q, 1)]=n_1.$$

To do so, we will consider some cases. The first case is if 
$$(q+1) \geq 2(p-1) - \lfloor \frac{|(p-1)-2|}{2} \rfloor.$$
This implies 
\begin{align*}
      q-1 &\geq 2(p-1) - \lfloor \frac{|(p-1)-2|}{2} \rfloor -2 \\
          &\geq 2(p-3) - \lfloor \frac{|(p-3)-2|}{2} \rfloor.
\end{align*}

Additionally, if $p$ is even, then 
\begin{align*}
      q-1 &\geq 2(p-1) - \lfloor \frac{|(p-1)-2|}{2} \rfloor -2 \\
          &\geq 2(p-2) - \lfloor \frac{|(p-2)-1|}{2} \rfloor.
\end{align*}

As we also know that 
$$q < 2p - \lfloor \frac{|p-2|}{2} \rfloor.$$
we have
\begin{align*}
    n_1 &= \frac{4}{5}(p+q -i) + 2 \lfloor \frac{i}{4} \rfloor\\
        &<  \frac{4}{5}(p+2p - \lfloor \frac{|p-2|}{2} \rfloor -i) + 2 \lfloor \frac{i}{4} \rfloor\\
\implies n_1 &\leq \frac{4}{5}(3p - \frac{p}{2} + 1.5 - i) + 2 \lfloor \frac{i}{4} \rfloor -1\\
        &\leq 2p + \frac{4}{5}(1.5 - i) + 2 \lfloor \frac{i}{4} \rfloor -1\\
\implies n_1 &< 2p\\
\implies n_1 &\leq 2(p-1).
\end{align*}

In these cases, consider the child of 
$CS_4(\mathcal{P}_4:p, 1, q, 1)$ which can be written as the sum of 
$CS_4(\mathcal{P}_4:p', 1, q-1, 1)$ and 
$CS_4(\mathcal{P}_4:p-p'-1, 1, 0, 1)$ where $q-1 \geq 2p' - \lfloor \frac{|p'-2|}{2} \rfloor$.  Notice that 
$$\eta[CS_4(\mathcal{P}_4:p-p'-1, 1, 0, 1)] = 0$$ 
according to equation~(\ref{eq nim 01q1}) and 
$$\eta[CS_4(\mathcal{P}_4:p', 1, q-1, 1)] = 2p'.$$
Observe that, this will cover all the even numbers up to 
$2(p-3)$, and up to $2(p-2)$ if $p$ is even. 

Therefore, we only need to think about the case when $p$ is odd and $n_1 = 2(p-1)$. In this case, $n_1$ is divisible by $4$. 
Consider the child obtained by connecting $x_{p-1}$ with $y_{q-2}$. 
This will imply that $CS_4(\mathcal{P}_4:p, 1, q, 1)$ has a child having nimber 
$$\eta[CS_4(\mathcal{P}_4:p-2, 1, q-3, 1)] \oplus \eta[CS_4(\mathcal{P}_4:1, 1, 2, 1)] = (n_1-4) \oplus 2 = n_1  - 2.$$ 
Thus we are done when 
$(q+1) \geq 2(p-1) - \lfloor \frac{|(p-1)-2|}{2} \rfloor$. 

\medskip

On the other hand, if $(q+1) < 2(p-1) - \lfloor \frac{|(p-1)-2|}{2} \rfloor$, then $p < q$ implies $p-1 < q+1$, 
by induction hypothesis we have 
$\eta[CS_4(\mathcal{P}_4:p-1, 1, q+1, 1)] = n_1$. 

Thus, in particular, for each $l \in \{2, 4, \cdots, n_1-2\}$, there exists 
a child of $CS_4(\mathcal{P}_4:p-1, 1, q+1, 1)$ having nimber equal to 
$l$. 
Observe that, this child must be a sum of two games of the form 
$CS_4(\mathcal{P}_4:p', 1, q', 1)$ and 
$CS_4(\mathcal{P}_4:p-p'-2, 1, q-q', 1)$. 
Due to equation~(\ref{eq balanced unbalanced}), we know that one of these children will satisfy the conditions of the
second line of the formula given in Theorem~\ref{th nimber}. 

Next consider the child of $CS_4(\mathcal{P}_4:p, 1, q, 1)$ obtained by connecting $x_{p'+1}$ and $y_{q'+1}$ 
can be expressed as a sum of 
the games 
$CS_4(\mathcal{P}_4:p', 1, q', 1)$ and 
$CS_4(\mathcal{P}_4:p-p'-1, 1, q-q'-1, 1)$. 
Due to equation~(\ref{eq balanced unbalanced}), we know that one of these children will satisfy the conditions of the second line of 
the formula given in Theorem~\ref{th nimber}. 

In fact as $p \neq q$, it is possible to assume 
without loss of generality that both 
$CS_4(\mathcal{P}_4:p-p'-2, 1, q-q', 1)$ and 
$CS_4(\mathcal{P}_4:p-p'-1, 1, q-q'-1, 1)$ 
satisfy the conditions of the
second line of the formula given in Theorem~\ref{th nimber}.

This implies 
\begin{align*}
    \eta[CS_4(\mathcal{P}_4:p-p'-2, 1, q-q', 1)] &= \frac{4}{5}(p+q-p'-q'-2 - i_1) + 2 \lfloor \frac{i_1}{2} \rfloor\\ 
    &= \eta[CS_4(\mathcal{P}_4:p-p'-1, 1, q-q'-1, 1)].
\end{align*}

Thus, we can conclude that this game, expressed as a sum of $CS_4(\mathcal{P}_4:p', 1, q', 1)$ and 
$CS_4(\mathcal{P}_4:p-p'-1, 1, q-q'-1, 1)$, 
which is a child of 
$CS_4(\mathcal{P}_4:p, 1, q, 1)$, has nimber equal to $l$. 
\end{proof}

\subsubsection{Concluding the proof}
\noindent \textit{Proof of Theorem~\ref{th nimber}}.  The
proof follows from the induction method presented above. \qed

\section{Nimber and winning strategy for $BS_2(\mathcal{P}_4:p,q)$}\label{sec nimber of BS}
Using Theorem~\ref{th nimber} we can calculate the nimbers of the game $BS_2(\mathcal{P}_4:p,q)$.

\begin{theorem}\label{cor nimber of bs}
The nimber of the game $BS_2(\mathcal{P}_4:p,q)$ is $0$ for all $p,q \geq 3$. 
\end{theorem}

\begin{proof}
Observe that in the game tree of 
$BS_2(\mathcal{P}_4:p,q)$, the root has exactly one child as the very first move is unique up to renaming of open tips. 
Moreover, that child will have a child 
which can be expressed as a sum of $CS_4(p,1,1,1)$ and $CS_4(1,1,q,1)$. As $\eta[CS_4(t_1,1,1,1)]\oplus \eta[CS_4(1,1,t_3,1)] = 0$ due to 
Theorem~\ref{th nimber}, the child of the root must have a non-zero nimber, and thus the root must have its nimber equal to $0$. 
\end{proof}

As a direct corollary, 
we can figure out which player has a winning strategy in the game $BS_2(\mathcal{P}_4:p,q)$.

\begin{corollary}
    The second player has winning strategy in the game 
    $BS_2(\mathcal{P}_4:p,q)$ for all $p,q \geq 3$.
\end{corollary}
     
\begin{proof}
    As $\eta[BS_2(\mathcal{P}_4:p,q)] = 0$ for all $p,q \geq 3$, the second player must have a winning strategy.
\end{proof}

\section{Conclusion}\label{sec conclusion}
\noindent (1) In Section~\ref{intro} we noticed that even though the game {\sc Brussels sprout} is defined on planar graphs, it can be played with the restriction of its resultant graph belonging to any hereditary graph family. Naturally, such a generalization of the game {\sc Sprout} can also be considered. It is noteworthy, that such generalizations of {\sc Sprout}, especially, its play on different surfaces, have been considered before~\cite{lemoine2008sprouts}.
It can be interesting to play both the games for hereditary graph classes such as: graphs having bounded maximum degree, graphs having bounded maximum average degree, planar graphs having girth at least $g$ (for some $g \geq 4$) to name a few.

\medskip
\noindent (2)  In Corollary~\ref{O_k}, we show that the winner of the game $BS_n(\mathcal{O}_k: t_1, t_2, \cdots, t_n)$ for $k \geq 1$, does not change even if we vary $k$. 
A similar phenomenon has been noticed in the game 
{\sc Sprout} through a computer check~\cite{applegatecomputer}. To elaborate, a computer check reported in~\cite{lemoine2012nimbers} confirmed that the winner of the game {\sc Sprout} (starting with $n$ dots) does not change for small values of $n$. 
It is possible that Theorem~\ref{th non-orientable surface} carries a theoretical insight of the underlying reasoning behind this phenomenon.

    \medskip
    
\noindent (3)  One of the major contributions of this article is to find out the nimber of the game $CS_4(\mathcal{P}_4:p, 1, q, 1)$ for all values of 
    $p, q \geq 0$ (see Section~\ref{sec nimber of circular}).  We used the method of strong induction for the proof, and we wonder if it is possible to attack Conjecture~\ref{conj sprouts} 
    using the same technique. 
    One of the major breakthroughs in applying this technique can come from understanding how it may be possible to write a partially played  {\sc Sprout} game as a sum of two other (similar) games. 

    \medskip
    
\noindent (4)  It may be interesting to explore a game where
    we start with some dots and some spots having various numbers of open tips. The reason why this makes sense and is relevant can be demonstrated via the following example. 
    \begin{figure}
    \centering
    \medskip

\begin{tikzpicture} [scale=0.5, transform shape]
\draw [thick] (2.5,0) ellipse (1.5 and 1);
\draw [thick] (10.5,0) ellipse (1.5 and 1);
\draw [thick] (6.5,-4) ellipse (1.5 and 1);
\node [circle, red, thick, draw, inner sep=0.1mm, minimum size=6mm] (b) at (1,0) {};
\node [circle, red, thick, draw, inner sep=0.1mm, minimum size=6mm] (b) at (12,0) {};
\node [circle, red, thick, draw, inner sep=0.1mm, minimum size=6mm] (b) at (6.5,-6) {};

\node[circle, draw,inner sep=0.1mm,minimum size=3mm, fill] (c1) at (1,0) {};
\node[circle, draw,inner sep=0.1mm,minimum size=3mm, fill] (c2) at (4,0) {};
\node[circle, draw,inner sep=0.1mm,minimum size=3mm, fill] (c3) at (9,0) {};
\node[circle, draw,inner sep=0.1mm,minimum size=3mm, fill] (c4) at (12,0) {};
\node[circle, draw,inner sep=0.1mm,minimum size=3mm, fill] (c5) at (5,-4) {};
\node[circle, draw,inner sep=0.1mm,minimum size=3mm, fill] (c6) at (8,-4) {};
\node[circle, draw,inner sep=0.1mm,minimum size=3mm, fill] (c7) at (6.5,0) {};
\node[circle, draw,inner sep=0.1mm,minimum size=3mm, fill] (c8) at (6.5,-3) {};
\node[circle, draw,inner sep=0.1mm,minimum size=3mm, fill] (c10) at (6.5,-6) {};

\draw[thick] (c2) -- (c3);
\draw[thick] (c7) -- (c8);
\path[thick] (c5) edge [bend right=100,looseness=2] (c6);


\end{tikzpicture}
\qquad
\begin{tikzpicture}
    \node (c7) at (0,0) {$=$};
    \node[circle, white,draw,inner sep=0.1mm,minimum size=3mm, fill] (c7) at (0,-2) {};
\end{tikzpicture}
\qquad
\begin{tikzpicture}
\node[circle, draw,inner sep=0.1mm,minimum size=3mm, fill] (c7) at (0,0) {};
\node[white,circle, draw,inner sep=0.1mm,minimum size=3mm, fill] (c9) at (-1,1) {};
\node[white,circle, draw,inner sep=0.1mm,minimum size=3mm, fill] (c10) at (1,1) {};
\node[white,circle, draw,inner sep=0.1mm,minimum size=1mm, fill] (c12) at (1,-2) {};

\node [circle, red, thick, draw, inner sep=0.1mm, minimum size=3mm] (b) at (-1,1) {};
\node [circle, red, thick, draw, inner sep=0.1mm, minimum size=3mm] (b1) at (1,1) {};
\node [circle, red, thick, draw, inner sep=0.1mm, minimum size=3mm] (b2) at (0,-1) {};
\node[circle, white,draw,inner sep=0.1mm,minimum size=2mm, fill] (c11) at (0,-1) {};
\draw[thick] (c9) -- (c7);
\draw[thick] (c10) -- (c7);
\draw[thick] (c11) -- (c7);
\end{tikzpicture}
    \caption{A configuration of {\sc Sprout} that can be replaced by a spot with three open tips.}
    \label{fig replacing C}
\end{figure}
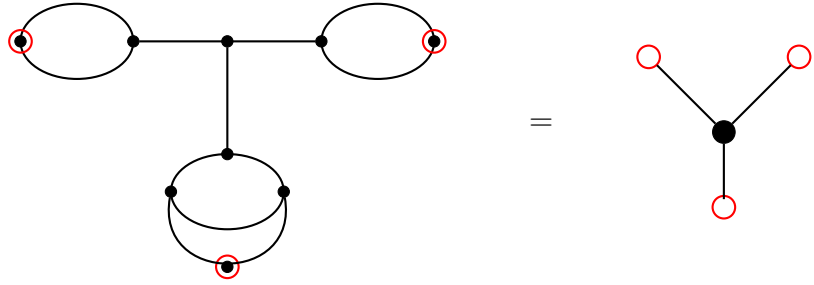
    A partially played 
    {\sc Sprout} game may contain a connected component $C$ as depicted in Fig.~\ref{fig replacing C}. In such a scenario, it is equivalent to considering the game obtained by replacing $C$ with a spot having $3$ open tips.

    \medskip
    
\noindent (5)  In this article we introduced the {\sc Circular sprout} game and studied it for 
    the family of triangle-free planar graphs. The version we studied had only two spots. Finding the nimber of this 
    turned out to be a challenging problem. Hence, finding the nimber of the {\sc Circular sprout} games with multiple spots for triangle-free planar graphs, and other hereditary families of graphs are natural open problems.

    \medskip
\noindent (6) Last but not the least, is it possible to consider a circular version of the game {\sc Sprout}? That may actually help us to understand the nature of certain subgames of {\sc Sprout}. 

\medskip 

\noindent
\textbf{Acknowledgements:}  We thank Prof. Gary MacGillivray for his valuable suggestions. This work is partially supported  by
the following projects:``MA/IFCAM/18/39'', ``SRG/
2020/001575'', ``MTR/2021/000858'', ``NBHM/RP-8 (2020)/Fresh'', and ``ASEAN 1000 PhD Fellowship''.
\small

\bibliographystyle{abbrv}
\bibliography{references}

\end{document}